\documentclass[12pt, a4paper]{article}

\usepackage[latin1]{inputenc}
\usepackage[english]{babel}

\usepackage{amssymb}
\usepackage{amsfonts}
\usepackage{amsmath}
\usepackage{amsthm}
\usepackage{epsfig}
\usepackage{dsfont}
\usepackage[usenames, dvipsnames]{xcolor}
\usepackage{verbatim}
\usepackage{latexsym}
\usepackage{setspace}

\usepackage{xcolor}
\usepackage{mathrsfs}
\usepackage{enumitem}
\usepackage{graphicx}
\usepackage{pgfplots}

\theoremstyle{plain}
\newtheorem{theorem}{Theorem}[section]
\newtheorem{corollary}[theorem]{Corollary}
\newtheorem{lemma}[theorem]{Lemma}
\newtheorem{proposition}[theorem]{Proposition}

\theoremstyle{definition}

\theoremstyle{definition}
\newtheorem{remark}[theorem]{Remark}

\newtheorem{example}[theorem]{Example}  
 
\allowdisplaybreaks[4]

\definecolor{see}{HTML}{0000AA}
\definecolor{doubt}{HTML}{AF0000}
\definecolor{note}{HTML}{600000}
\definecolor{theo1}{HTML}{0000D0}
\definecolor{theo2}{HTML}{000060}
\definecolor{defi1}{HTML}{006000}
\definecolor{defi2}{HTML}{008800}

\newcommand{\C}{\mathcal{C}}

\newcommand\diff{\mathop{}\!\mathrm{d}}
\newcommand\e{\mathop{}\!\mathrm{e}}
\newcommand{\R}{\mathbb{R}}
\newcommand{\E}{\mathbb{E}}

\renewcommand{\P}{\mathbb{P}}

\newcommand{\N}{\mathbb{N}}

\newcommand{\Z}{\mathbb{Z}}
\renewcommand{\l}{\left}
\renewcommand{\r}{\right}

\newcommand{\NN}{\mathbb{N}}

\newcommand{\RR}{\mathbb{R}}
\newcommand{\PP}{\mathbb{P}}

\newcommand{\EE}{\mathbb{E}}

\newcommand{\cC}{\mathcal{C}}

\newcommand{\bone}{\mathds{1}}

\begin{document}
	\title{On $q$-scale functions of spectrally negative compound Poisson processes} 
	\author{Anita Behme\thanks{Technische Universit\"at
			Dresden, Institut f\"ur Mathematische Stochastik, Zellescher Weg 12-14, 01069 Dresden, Germany, \texttt{anita.behme@tu-dresden.de} and \texttt{david.oechsler@tu-dresden.de}, phone: +49-351-463-32425, fax:  +49-351-463-37251.}\; and David Oechsler$^\ast$}
	\date{\today}
	\maketitle
	
	\vspace{-1cm}
	\begin{abstract}
		Scale functions play a central role in the fluctuation theory of spectrally negative Lévy processes. For spectrally negative compound Poisson processes with positive drift, a new representation of the $q$-scale functions in terms of the characteristics of the process is derived. Moreover, similar representations of the derivatives and the primitives of the $q$-scale functions are presented. The obtained formulae for the derivatives allow for a complete exposure of the smoothness properties of the considered $q$-scale functions. Some explicit examples of $q$-scale functions are given for illustration.
	\end{abstract}
	
	

	\section{Introduction}\label{S0}
	\setcounter{equation}{0}

	Among Lévy processes those with spectrally negative jump measures play a special role in applied mathematics. On the one hand this is due to the fact that for many real life scenarios in queuing, risk, finance, etc. it is intrinsically reasonable to allow only jumps in one direction. 	On the other hand spectrally negative Lévy processes offer mathematical advantages when considering exit problems and the related \(q\)-scale functions. Recall that for a spectrally negative Lévy process \((L_t)_{t\geq 0}\) with Laplace exponent \(\psi\) and for \(q\geq0\) the \(q\)-scale function \(W^{(q)}:\R_+\to\R\) is defined as the unique function such that
	\begin{equation}\label{eq-def-Wq}
	\int_0^\infty \e^{-\beta x} W^{(q)}(x) \diff x = \frac1{\psi(\beta)-q}
	\end{equation}
	holds for all \(\beta>\sup\{y:\psi(y)=q\}=:\Phi(q)\). These  \(q\)-scale functions get their name from the related scale functions of regular diffusions and indeed, just as for their counterparts, many fluctuation identities of spectrally one-sided Lévy processes may be expressed in terms of \(q\)-scale functions, see e.g. \cite{Bertoin97, emery, rogers1, rogers2}, or \cite{doney, kuznetsov2011theory} and references therein.
	
	However, being defined as an inverse Laplace transform,  analytic or even  closed-form expressions of $q$-scale functions are available only in exceptional cases; see e.g. \cite{hubalek2010} for a collection of such cases. 
	Although there are practical ways to evaluate \(q\)-scale functions numerically as exposed e.g. in \cite{kuznetsov2011theory}, or to approximate them as done e.g. in \cite{Egami2010PhasetypeFO}, it is not only of theoretical interest to expand the library of cases with analytic expressions. As they are intermediate objects in the sense that we are not only interested in \(q\)-scale functions per se, but merely in expressions derived from them, evaluating \(q\)-scale functions numerically both hinders error analysis and amplifies error propagation, e.g. when the jump measure must be estimated which is basically the case in most applications. 
	
	In this article, after setting the stage with some preliminaries given in Section \ref{S2}, in Section \ref{S3} we provide and prove our main Theorem \ref{main}. This states that the \(q\)-scale functions of a spectrally negative compound Poisson process with drift $c >0$, intensity $\lambda>0$ and jump distribution $\Pi(-\diff s)$ can be represented as
		\begin{align}\label{mainform}
	W^{(q)}(x)=\frac1c \sum_{k=0}^{\infty} \int_{[0,x]} g_k(s,x) ~\Pi^{*k}(\diff s), \quad x\geq 0,
	\end{align}
	where 
	\begin{equation}\label{g}
	g_k(s,x)=\frac{(\lambda/c)^k }{k!}\e^{-\frac{q+\lambda}{c}(s-x)}(s-x)^k,\quad s, x\geq 0,
	\end{equation} 
	and $\Pi^{\ast k}$ denotes the $k$-th convolution power of $\Pi$.
	Note that \eqref{mainform} breaks down into a closed-form expression if the jump measure is discrete. In that particular case our result may be regarded as a generalization of the \(0\)-scale function
	\begin{align}\label{Erlang}
	W^{(0)}(x)=\frac{1}{c}\sum_{s=0}^{\lfloor x \rfloor} \frac{\e^{-\frac\lambda c (s-x)} }{s!}\l(\frac{\lambda}{c}\r)^s\l(s-x\r)^s, \quad x\geq 0,
	\end{align}
	of \(L_t=c t - N_t\), $t\geq 0$, where \(c>0\) and \(N_t\) is a Poisson point process with intensity \(\lambda>0\). Formula \eqref{Erlang} can be found in \cite{hubalek2010} or \cite{asmussenalbrecher} and dates all the way back to Erlang \cite{Erlang1909}, even though the notions of neither scale functions nor Lévy processes were present then. \\
	At first glance the representation \eqref{mainform} resembles the form of the $0$-scale function one obtains via the so-called Pollaczeck-Khintchine formula (cf. \cite[IV (2.2)]{asmussenalbrecher}, \cite[Thm. 1.8]{Kyprianou2014} or \cite[Thm. 1.3]{KyprianouGS}), a prominent statement in actuarial mathematics. Namely, as long as the jumps $\{\xi_i,i\in \NN\}$ of the considered compound Poisson process have finite mean $\mu$ such that $\lambda \mu>c$, one can show that (see e.g. \cite[Eq. (4.12)]{KyprianouGS})
	\begin{equation*}
	W(x)= \frac{1}{c} \sum_{k=0}^\infty \l( \frac{\lambda \mu}{c} \r)^k \Pi_I^{\ast k}(x), \quad x\geq 0,
	\end{equation*}
	where $\Pi_I(\diff x)=\mu^{-1}\Pi((x,\infty)) \diff x$ is the integrated tail function of the jump distribution. However, besides of being valid for all $q\geq 0$, \eqref{mainform} does not need an assumption like $\lambda \mu>c$. Moreover it relies on convolutions of the jump distribution itself, instead of its integrated tail function.
	
	For many purposes it is essential to know not only the \(q\)-scale functions \(W^{(q)}\) but also their primitives and - in case they exist - their derivatives, see  e.g. \cite[Sec. 1]{kuznetsov2011theory} for a collection of applications that stem from various areas of applied probability theory. 
	The question of smoothness is of theoretical interest as well (cf. \cite{chankypsavov, kyprivsong} and \cite[Sec. 3.5]{kuznetsov2011theory} for more on that topic) as one can interpret \(W^{(q)}\) as an eigen-function of the infinitesimal generator \(\mathcal{A}\) of \((L_t)_{t\geq0}\), i.e.
	\[
	(\mathcal{A}-q)W^{(q)} = 0
	\]
	in some sense. However, this equation is to be treated cautiously as it is not clear whether \(W^{(q)}\) is in the domain of \(\mathcal{A}\) or not. In Section \ref{S41} we present representation formulae for all derivatives of $W^{(q)}$ possible. These immediately confirm Doney's conjecture  \cite[Conj. 3.13]{kuznetsov2011theory} which states in our setting (i.e. for spectrally negative compound Poisson processes ) that
	\begin{align}\label{equivalence}
	W^{(q)}\in\C^{n+1}(0,\infty) \Leftrightarrow \overline{\Pi}\in\C^{n}(0,\infty)
	\end{align}
	for \(n\in\N\), where \(\overline{\Pi}(x)=\Pi((-\infty, x])\) is the cumulative distribution function of \(\Pi\).  Moreover, we show that for absolutely continuous measures $\Pi$ the equivalence \eqref{equivalence} is a local property, i.e. \( (W^{(q)})^{(n+1)}(x_0)\) exists (and is continuous) for \(x_0\in\R_+\) if and only if \(\overline{\Pi}^{(n)}(x_0)\) exists (and is continuous). With this we slightly improve the corresponding result \cite[Thm. 3]{chankypsavov}. \newline
	A representation formula for the primitive of \(W^{(q)}\)  is derived in Section \ref{S42}, while in Section \ref{S43} we shall rewrite our representations of $W^{(q)}$, its first derivative, and its primitive in the form of $\EE^x[f_{q,\lambda, c}(L_1, N_1)]$ for appropriate functions $f_{q,\lambda, c}$ and with $N_1$ being the number of jumps of $(L_t)_{t\geq 0}$ up to time $1$. 
	
	The final Section \ref{S5} is devoted to some new examples of explicitely computable q-scale functions.

	\section{Preliminaries} \label{S2}
	\setcounter{equation}{0}
	
	Throughout the paper $(L_t)_{t\geq 0}$ denotes a spectrally negative Lévy process with Laplace exponent
	\begin{equation}\label{eq-laplaceexp}
\psi(u)= \log\EE[\e^{uL_1}] = \tilde{c} u + \frac12 \sigma^2 u^2 + \int_{(-\infty,0)} (\e^{ux}-1- ux\mathds{1}_{x> -1})\tilde{\Pi}(\diff x),
\end{equation}
with $\tilde{c}\in\RR$, $\sigma^2\geq 0$ and a Lévy measure $\tilde{\Pi}$ concentrated on $(-\infty,0)$.

We first note that rescaling the given Lévy process corresponds to a rescaling of the associated $q$-scale functions, as shown by the following lemma, which follows immediately from the definition of the $q$-scale functions in \eqref{eq-def-Wq} and the form of the Laplace exponent \eqref{eq-laplaceexp}.

\begin{lemma}\label{scaling}
	Let \(\varepsilon>0\). For any \(q\geq0\) the \(q\)-scale function of \((L_t)_{t\geq 0}\) is given by
	\[
	W^{(q)}(x)=\frac1\varepsilon \hat{W}^{(q)} \l(\frac x\varepsilon\r), \quad x\geq 0,
	\]
	where \(\hat{W}^{(q)}\) is the \(q\)-scale function of the Lévy process $(\hat{L}_t)_{t\geq 0}$ given by \(\hat{L}_t:=\frac1\varepsilon L_t,~t\geq0\).
\end{lemma}
	
In a similar manner the next lemma allows to adjust the drift by proper rescaling of the scale functions. It is an immediate consequence of the definition of the $q$-scale function in \eqref{eq-def-Wq} and the fact that $\log \EE[\e^{uL_t}]=t\psi(u)$.
	\begin{lemma}\label{scaling1}
		Assume $\tilde{c}> 0$. Then for any \(q\geq0\) the \(q\)-scale function of \((L_t)_{t\geq 0}\) is given by
		\[
		W^{(q)}(x)= \frac{1}{\tilde{c}} \hat{W}^{( q/\tilde{c})}(x),\quad x\geq 0,
		\]
		where \( \hat{W}^{(q/\tilde{c})}(x), x\geq 0, \) is the \(q/\tilde{c}\)-scale function of the subordinated process $(\hat{L}_t)_{t\geq 0}$ given by \(\hat L_t := L_{t/\tilde{c}},~t\geq0\).
	\end{lemma}

We are especially interested in spectrally negative compound Poisson processes for which \eqref{eq-laplaceexp} simplifies to 
\begin{equation*}
\psi(u)= \log\EE[\e^{uL_1}] = c u + \lambda \int_{(0,\infty)} (\e^{-ux}-1)\Pi(\diff x),
\end{equation*}
where we call $c$ the \emph{drift}, $\lambda\geq 0$ the \emph{intensity}, and $\Pi(\diff x)$ the \emph{jump measure} describing the negative of the jumps of $(L_t)_{t\geq 0}$. We will thus consider 
\begin{equation}\label{eq-specialtype}
L_t=c t- \sum_{i=1}^{N_t} \xi_i, \quad t\geq 0,
\end{equation}
for a Poisson process $(N_t)_{t\geq 0}$ with intensity $\lambda>0$ and i.i.d. non-negative random variables $\xi_i,i\geq1$ with distribution $\Pi$. In view of Lemma \ref{scaling1} we will often set $c=1$.

Whenever we consider a shifted version of the Lévy process starting in $x\in\RR$ we use the notation $\PP^x$ for the induced probability measure and $\EE^x$ for the corresponding expectation. For $x=0$ we set $\PP:=\PP^0$ and $\EE=\EE^0$. 
We abbreviate the natural numbers including $0$ as $\NN=\{0,1,2,\ldots\}$, and the real half-line as $\RR_+=[0,\infty)$. 
Integrals of the form $\int_a^b$ are to be understood as $\int_{[a,b]}$. Partial derivatives are abbreviated as $\partial_x:=\frac{\partial}{\partial x}$, while $\partial_{x,\pm}$ denote left- and right derivatives with respect to $x$. In case of only one variable, we may omit the $x$ and simply use $\partial$, or $\partial_{\pm}$, for the (directional) derivative. 
Finally, throughout the paper we shall use the standard floor and ceiling functions $\lfloor x \rfloor := \max \{y\in\Z~:~y\leq x \}$ and $\lceil x \rceil := \min \{y\in\Z ~:~ y\geq x\}$.

\section{A representation of \(W^{(q)}\)}\label{S3}
\setcounter{equation}{0}

As mentioned in the introduction, our main result reads as follows.

	\begin{theorem}\label{main}
	For any \(q\geq0\) the \(q\)-scale function of the spectrally negative compound Poisson process in \eqref{eq-specialtype} is given by \eqref{mainform}.
	\end{theorem}
	
To prove this theorem in the forthcoming subsections, we choose a dense subclass of processes for which we can explicitly compute the scale functions. The expressions we obtain will then be extended to the general setting of the theorem by approximation arguments. 

\subsection{A recursion formula}

	We start by proving a recursion formula for $q$-scale functions of spectrally negative compound Poisson processes whose jump measures are bounded away from zero.
	
	\begin{lemma}\label{recursivformula1}
		Let \((L_t)_{t\geq 0}\) be a spectrally negative compound Poisson process as in \eqref{eq-specialtype} with drift $c=1$. Let further \(J:=\mathrm{supp}(\Pi) \subset [1,\infty)\) and fix \(q\geq0\). Then the \(q\)-scale function \(W^{(q)}\) of \((L_t)_{t\geq 0}\) fulfils
		\begin{align}\label{recurse}
		W^{(q)}(x) = \e^{(q+\lambda) x} \l( \e^{-(q+\lambda)\lfloor x\rfloor} W^{(q)}( \lfloor x \rfloor) - \lambda\int_J  \int^x_{\lfloor x\rfloor} \e^{-(q+\lambda) z}W^{(q)}(z-y) \diff z ~\Pi(\diff y)\r)
		\end{align}
		for \(x\geq0\), if we set \(W^{(q)}(x)=0\) for \(x<0\). Moreover, for any solution \(V\) of \eqref{recurse} with \(V(x)=0\) for \(x<0\) it holds 
		\begin{align}\label{unique}
		V(x)=V(0)W^{(q)}(x)
		\end{align}
		for all \(x\in\R\).
	\end{lemma}
	
	\begin{proof}
		Note first that, if some function solves \eqref{recurse}, then so does any multiple of it. \\		
		To prove \eqref{recurse}, fix \(q\geq 0\) and let \(a>0\). Define the first passage times
		\begin{align*}
		\tau_a^+&:= \inf\{t>0: L_t\geq a\},\\
		\tau_0^-&:= \inf\{t>0: L_t\leq0\}.
		\end{align*}
		It is well known (cf. \cite[Thm. 1.2]{kuznetsov2011theory}) that
		\begin{equation}\label{scaleviapassage}
		\E^x[\e^{-q\tau_a^+}\bone_{\{\tau_a^+<\tau_0^-\}}] = \frac{W^{(q)}(x)}{W^{(q)}(a)}.
		\end{equation}
		Scale functions are defined on \([0,\infty)\) and in this particular setting we are able to compute them directly on \([0,1]\) due to \(J\subset [1,\infty)\).\\
		To do so, let \(a=1\) and \(x\in(0,1)\) arbitrary. 
		Under \(\P^x\) the event \(\{\tau_1^+<\tau_0^-\}\) is equivalent to the event that no jump occurs until \(T={1-x}\), the time it takes for the pure drift to exit the interval. Thus \(\P^x (\tau_1^+<\tau_0^-)= \e^{\lambda (x-1)}\) due to the number of jumps in an interval \([0,T]\) following a Poisson distribution with parameter \(\lambda T\). In that case \(\tau_1^+={1-x}\) and we find
		\begin{align}
		\E^x[\e^{-q\tau_1^+} \bone_{\{\tau_1^+<\tau_0^-\}}] &=  \e^{ q (x-1)}\P^x (\tau_1^+<\tau_0^-) = \e^{(q+\lambda)(x-1)}.\label{firstinterval}
		\end{align}
		Thus for \(x\in(0,1)\), 
		\[
		W^{(q)}(x) = \e^{(q+\lambda)(x-1)} W^{(q)}(1),
		\]
		and hence, \eqref{recurse} is fulfilled on \((0,1)\) because the integral in \eqref{recurse} vanishes and \(\e^{-(q+\lambda)}W^{(q)}(1)\) is just a constant factor.  Since scale functions are continuous (cf. \cite[Lem. 2.3]{kuznetsov2011theory}) on \([0,\infty)\) the formula holds true for \(x=0\) and \(x=1\) as well. 
		However, from \cite[Lem. 3.1]{kuznetsov2011theory}, we know that \(W^{(q)}(0)=\frac{1}{c}\), and as we agreed on \(c=1\), we have
		\begin{align}\label{on1}
		W^{(q)}(x)=\e^{(q+\lambda)x}, \quad x\in[0,1].
		\end{align}
	
		Clearly, if we prove \eqref{recurse} for \(x\in\R_+\setminus\N\), we may close any gap by continuity. Therefore let \(x\in\R_+\setminus\N\) with \(x>1\).
		As we want to use \eqref{scaleviapassage} again, we are interested in the expression \(\E^x[\e^{-q\tau_{\lceil x\rceil}^+}\bone_{\{\tau_{\lceil x\rceil}^+<\tau_0^-\}}] \) which can be written as
		\begin{equation}\label{onrwithoutn}
		\E^x[\e^{-q\tau_{\lceil x\rceil}^+}\bone_{\{\tau_{\lceil x\rceil}^+<\tau_0^-\}}]  = \E^x[\e^{-q\tau_{\lceil x\rceil}^+}\bone_{\{\tau_{\lceil x\rceil}^+<\sigma\}}]+
		\E^x[\e^{-q\tau_{\lceil x\rceil}^+}\bone_{\{\tau_{\lceil x\rceil}^+>\sigma\}\cap\{L_{\sigma}>0\}\cap\{\hat\tau_{\lceil x\rceil}^+<\hat\tau_{0}^-\}}],
		\end{equation}
		where \(\sigma\) denotes the time of the first jump of $(L_t)_{t\geq 0}$ and 
		\begin{align*}
		\hat\tau^+_{\lceil x\rceil} &:=\inf\{t\geq\sigma: L_t\geq\lceil x\rceil \} - \sigma, \\
		\hat\tau_{0}^- &:= \inf\{t\geq\sigma: L_t\leq0 \} - \sigma
		\end{align*}				
		denote the respective first passage times after the first jump. The first term in \eqref{onrwithoutn} is 
		\[
		\E^x[\e^{-q\tau_{\lceil x\rceil}^+}\bone_{\{\tau_{\lceil x\rceil}^+<\sigma\}}]=\E^{x-\lfloor x\rfloor}[\e^{-q\tau_1^+}\bone_{\{\tau_1^+<\sigma\}}]=\E^{x-\lfloor x\rfloor}[\e^{-q\tau_1^+}\bone_{\{\tau_1^+<\tau^-_0\}}]=
		\e^{(q+\lambda)(x-\lceil x\rceil)}.
		\]
		due to the space homogeneity of Lévy processes, $J\subseteq[1,\infty)$ and formula \eqref{firstinterval}. For the second expression in \eqref{onrwithoutn} note that on \(\{\tau_{\lceil x\rceil}^+>\sigma\}\) it holds \(\tau_{\lceil x\rceil}^+ = \sigma +  \hat\tau_{\lceil x\rceil}^+\) by definition and \(\sigma=\tau_{\lfloor x \rfloor}^-\) since \(J\subset [1,\infty)\). We condition on $\sigma\sim \mathrm{Exp}(\lambda)$ and $\xi_1=-\Delta L_\sigma$ and since under $\PP^x$ we have $L_\sigma=x+\sigma-\xi_1$, we obtain
		\begin{align*}
		\E^x\l[\e^{-q\tau_{\lceil x\rceil}^+}\bone_{\{\tau_{\lceil x\rceil}^+>\sigma\}\cap\{L_{\sigma}>0\}\cap\{\hat\tau_{\lceil x\rceil}^+ < \hat\tau_{0}^-\}}\r]
		&= 	\E^x\l[\E^x\l[\e^{-q\tau_{\lceil x\rceil}^+} \bone_{\{\tau_{\lceil x\rceil}^+>\sigma\}\cap\{L_{\sigma}>0\}\cap\{\hat\tau_{\lceil x\rceil}^+ < \hat\tau_{0}^-\}}\middle| \sigma, \xi_1 \r] \r]\\
		&= \E^x\l[\e^{-q\sigma} \bone_{\{\tau_{\lceil x\rceil}^+>\sigma\}\cap\{L_{\sigma}>0\}} \E^{L_{\sigma}}\l[\e^{-q\hat\tau_{\lceil x\rceil}^+} \bone_{\{\hat\tau_{\lceil x\rceil}^+ <\hat\tau_{0}^-\}}\r] \r] \\
		&= \E^x\l[\e^{-q\sigma} \bone_{\{\tau_{\lceil x\rceil}^+>\sigma\}\cap\{L_{\sigma}>0\}} \frac{W^{(q)}(L_{\sigma})}{W^{(q)}(\lceil x\rceil)} \r] \\
		&= \frac1{W^{(q)}({\lceil x\rceil})} \lambda\int_J \int_0^{{\lceil x\rceil}-x}\e^{-(q+\lambda)t}W^{(q)}(x+ t-y)\diff t ~ \Pi(\diff y).
		\end{align*}
		 Note that in the above the indicator function \(\mathds{1}_{\{L_\sigma>0\}}\) can be omitted due to \(W^{(q)}(x)=0\) if \(x<0\). \newline
		Via \eqref{scaleviapassage} we obtain now by combination of the two terms in \eqref{onrwithoutn} 
		\begin{align*}
		W^{(q)}(x)&= \E^x[\e^{-q\tau_{\lceil x\rceil}^+}\bone_{\{\tau_{\lceil x\rceil}^+<\tau_0^-\}}] \cdot W^{(q)}({\lceil x\rceil})\\
		&=\e^{(q+\lambda)(x-\lceil x\rceil)}W^{(q)}(\lceil x\rceil) +  \lambda\int_J \int_0^{\lceil x\rceil-x}\e^{-(q+\lambda)t}W^{(q)}(x+ t-y)\diff t~ \Pi(\diff y),\\
		&=\e^{(q+\lambda)(x-\lceil x\rceil)}W^{(q)}(\lceil x\rceil) + \lambda\int_J\int_{x}^{\lceil x\rceil}\e^{(q+\lambda)(x-z)}W^{(q)}(z-y)\diff z ~ \Pi(\diff y).\end{align*}
		Recall that \(x\in\R_+\setminus\N\) and choose a decreasing sequence \(\{x_n\}_{n\in\N}\) with $x_n\to\lfloor x\rfloor$ as $n\to\infty$. Taking the limit we obtain
		\[
		W^{(q)}(\lfloor x\rfloor)=\e^{-(q+\lambda)}W^{(q)}({\lceil x\rceil}) + \lambda\int_J \int_{\lfloor x\rfloor}^{\lceil x\rceil}\e^{({q+\lambda})(\lfloor x\rfloor-z)}W^{(q)}(z-y)\diff z~\Pi(\diff y)
		\]
		by continuity. If we now combine the last two formulae we get
		\begin{align*}
		W^{(q)}(x)=\e^{(q+\lambda) x}\l(\e^{-(q+\lambda)\lfloor x\rfloor}W^{(q)}(\lfloor x\rfloor)-\lambda\int_J \int_{\lfloor x\rfloor}^x \e^{-(q+\lambda) z}W^{(q)}(z-y)\diff z ~ \Pi(\diff y)\r),
		\end{align*}
		the desired result. As mentioned above this implies \eqref{recurse} for \(x\in\N\) as well.\\		
		Finally, given a solution $V$ of \eqref{recurse} with $V(x)=0$ for $x<0$, it follows directly from \eqref{recurse}, that $V(x)=e^{(q+\lambda)x} V(0)= W^{(q)}(x) V(0)$ on $[0,1)$. Recursively \eqref{unique} now follows for all $x\in\RR$.
	\end{proof}

	Formula \eqref{recurse} can be solved inductively for lattice distributed jumps. Before doing so in the next section we show a simple corollary of the above which simplifies our notation. 
	
	\begin{corollary}\label{auxiliary}
		In the setting of Lemma \ref{recursivformula1} for all \(q\geq0\) there exist continuous functions \(w_q: \R\to\R\) such that for \(x\in\R\)
		\begin{align}\label{ansatz}
		W^{(q)}(x)=\e^{(q+\lambda) x}w_q(x),
		\end{align}
		and \(w_q\) fulfils the recursive formula
		\begin{align}\label{recurse2}
		w_q(x)=w_q(\lfloor x\rfloor) - \lambda\int_J \e^{-(q+\lambda) y}\int^x_{\lfloor x\rfloor} w_q(z-y) \diff z ~\Pi(\diff y)
		\end{align}
		for \(x\geq0\), while \(w_q(x)=0\) for \(x<0\). In particular
			\begin{equation}\label{recurse_start}
		w_q(x)=1 \quad \text{for } x\in[0,1].
		\end{equation}
	\end{corollary}
	\begin{proof}
		We obtain \eqref{recurse2} by plugging \eqref{ansatz} into formula \eqref{recurse}. Formula \eqref{recurse_start} follows from \eqref{on1}.
	\end{proof}

\subsection{Lattice-distributed jumps}

The recursion \eqref{recurse2} becomes particularly easy for lattice-distributed jumps as the outer integral breaks down into a finite sum. Here, by call $\Pi$ a \emph{(positive) lattice distribution}, if $\Pi$ is a discrete probability measure of the form
\[\Pi (\diff x) =\sum_{y\in J}   p_y \delta_y (\diff x)
\]
with \(\sum_{y\in J}  p_y = 1\) and  \(J:=\mathrm{supp}(\Pi)\subset \varepsilon\N\setminus\{0\}\) for some \(\varepsilon>0\). The parameter \(\varepsilon\) is called the \emph{step} of the lattice measure. 

\begin{remark}
Other definitions of lattice measures exist in the literature. A more general lattice would be 
\[
\mathrm{supp}(\Pi)=a+\varepsilon\Z,
\]
for some \(a\in\R\) and \(\varepsilon>0\). Also general lattice distributions need not be finite measures. However, we restrict ourselves to the cases where the definition above applies.
\end{remark}

By Lemma \ref{scaling} and Lemma \ref{scaling1} it suffices to consider \(\varepsilon=1\) and we shall do so for the moment. By Corollary \ref{auxiliary} the functions $w_q$ thus fulfil
\begin{equation}\label{recurse3} \begin{aligned}
w_q(x)&=w_q(\lfloor x\rfloor) -  \lambda\sum_{y\in J}   p_y \e^{-(q+\lambda) y}\int^x_{\lfloor x\rfloor} w_q(z-y) \diff z, \quad x\geq 0,\\
\text{with} \quad w_q(x)&=1, \quad x\in[0,1], \quad \text{and} \quad w_q(x)=0, \quad x<0.
\end{aligned}\end{equation}
We shall use the following slight abuse of notation for the sake of shortness. Let \(v\in\N^J\), i.e. \(v: J \to \N , y\mapsto v_y\) is a function which maps every possible jump height on a non-negative integer. Then we define for any \(v\in\N^J\)
\begin{align*}
|v|:=\sum_{y\in J} v_y, \quad \text{and} \quad 
\langle v, J \rangle := \sum_{y\in J} v_y y.
\end{align*}
We will interpret \(v\) as a vector that keeps track of how many jumps of which size occur. Then \(|v|\) is the total number and \(\langle v,J \rangle\) the accumulated size of all occurring jumps.

\begin{lemma}\label{lattice_explicit}
Assume $\Pi=\sum_{y\in J}   p_y \delta_y $ is a positive lattice distribution with $J\subset \NN\setminus\{0\}$ and set \(m_y={\lambda  p_y} \e^{-(q+\lambda) y}, y\in J\). Then the unique solution \(w_q\) to \eqref{recurse3} is given by
\begin{align}\label{latsol}
w_q(x)= \sum_{\substack{v\in\N^{J} \\ \langle v, J\rangle\leq\lfloor x\rfloor}} \l(\prod_{\ell\in J} \frac{m_\ell^{v_\ell}}{v_\ell !}\r)\Big(\langle v, J\rangle-x\Big)^{|v|},~x\geq0.
\end{align}
\end{lemma}
\begin{proof}
First note that both the sum and the product in \eqref{latsol} are always finite: Since \(J\) is bounded away from zero, for any fixed $x\geq 0$ the condition on \(v\in\N^{J}\) implies that only those \(v\in\N^{J}\) with \(|v|\leq \lfloor x\rfloor\) may appear, and the entries of \(v\) being in \(\N\) implies that merely finitely many of them are non-zero. \newline
Consider \(x\in[0,1)\). The sum in \eqref{latsol} then consists of only one summand, namely the one generated by \(v\equiv 0\), and we obtain \(w_q(x)=1\). The second summand is generated by \(v=(1,0,\ldots)\) and appears for \(x\geq1\). However, it vanishes for \(x=1\) as \(\langle v, J\rangle-x=0\) and, hence, \(w_q(1)=1\) as well which coincides with \eqref{recurse3}. Indeed, any new summand is zero at its first appearance and this guarantees continuity of \(w_q\). \newline
Now assume \eqref{latsol} is true on \([0,n]\) for some \(n\in\N\) and let \(x\in(n,n+1)\). By \eqref{recurse3} we find
\[
w_q(x)=w_q(n) - \sum_{y\in J} m_y \int^x_{n} \sum_{\substack{v\in\N^{J} \\ \langle v, J\rangle\leq \lfloor z-y \rfloor}} \l(\prod_{\ell\in J} \frac{m_\ell^{v_\ell}}{v_\ell !}\r)\Big(\langle v, J\rangle-z+y\Big)^{|v|} \diff z.
\]
To swap the integral with the sum to its right we fix \(y\in J\) for a moment. Since we integrate over \(z\in[n,x]\) and \(y\in J \subset \N\) it holds \(\lfloor z-y\rfloor =  n-y\). Thus, the summation area is independent of \(z\) and we can compute further
\begin{align*}
w_q(x)&=w_q(n) -  \sum_{y\in J} m_y \sum_{\substack{v\in\N^{J} \\ \langle v, J\rangle\leq  n-y}}\l(\prod_{\ell\in J} \frac{m_\ell^{v_\ell}}{v_\ell !}\r) \int^{x}_{n}  \Big(\langle v, J\rangle-z+y\Big)^{|v|} \diff z \\
&=w_q(n) + \sum_{y\in J} m_y \sum_{\substack{v\in\N^{J} \\ \langle v, J\rangle\leq n-y}}\l(\prod_{\ell\in J} \frac{m_\ell^{v_\ell}}{v_\ell !}\r)
  \l(\frac{(\langle v, J\rangle-x+y)^{|v|+1} }{|v|+1}-\frac{(\langle v, J\rangle-n+y)^{|v|+1} }{|v|+1}\r).
\end{align*}
Denoting by \(e_y\in\N^{J}\) the unique element with $|e_y|=1$ and  \(\langle e_y,J\rangle= y \) we can write
\begin{align*}
w_q(x)&=w_q(n) + \sum_{y\in J}\sum_{\substack{v\in\N^{J} \\ \langle v+e_y, J\rangle\leq n}}\l(\prod_{\ell\in J} \frac{m_\ell^{(v+e_y)_\ell}}{v_\ell !}\r)
  \l(\frac{(\langle v+e_y, J\rangle-x)^{|v+e_y|} }{|v+e_y|}-\frac{(\langle v+e_y, J\rangle-n)^{|v+e_y|} }{|v+e_y|}\r).
\end{align*}
 Although not obvious at first, it is possible to combine the two sums. To achieve that we need to count how often each summand appears. Recall that only finitely many entries of the appearing  $v$ are non-zero. Thus by first multiplying with \(\frac{|v|!}{|v|!} = |v|!\cdot \frac{|v+e_y|}{|v+e_y|!}\) we may insert a multinomial coefficient and obtain
\begin{align*}
w_q(x)=w_q(n) + \sum_{y\in J}\sum_{\substack{v\in\N^{J} \\ \langle v+e_y, J\rangle\leq n}}&{|v| \choose v_1,v_2, \cdots}\l(\frac{\prod_{\ell\in J}^\infty m_\ell^{(v+e_y)_\ell}}{|v+e_y|!}\r)\cdot\\
 & \cdot\l((\langle v+e_y, J\rangle-x)^{|v+e_y|}-(\langle v+e_y, J\rangle-n)^{|v+e_y|} \r),
\end{align*}
where we set $v_x=0$ for $x\in \NN\setminus J$.
Further, set
\[
f(v)= \l(\frac{\prod_{\ell\in J} m_\ell^{v_\ell}}{|v|!}\r)\l((\langle v, J\rangle-x)^{|v|}-(\langle v, J\rangle-n)^{|v|} \r), \quad v\in\N^{J},
\]
then
\begin{align*}
w_q(x)=w_q(n) + \sum_{y\in J}\sum_{\substack{v\in\N^{J} \\ \langle v+e_y, J\rangle\leq n}}{|v| \choose v_1,v_2, \cdots}f(v+e_y).
\end{align*}
Observe that \(\bar v\in\N^{J}\) appears in the argument of \(f\) if it fulfils \(\langle \bar v,J\rangle \leq n\) and if there exists a vector \(v\in\N^{J}\) such that \(\bar v = v+e_y\) for some \(y\in J\). 
Thus  \(\bar v\equiv 0\) is not part of the sum due to it being obviously the only vector where there are no \(v\in\N^{J}\) and \(y\in J\) such that \(v+e_y\equiv 0\), and hence,
\begin{align*}
w_q(x)&=w_q(n) + \sum_{y\in J}\sum_{\substack{\bar{v}\in\N^{J} \\ 1\leq \langle \bar{v}, J\rangle\leq n}}{|\bar{v}-1| \choose  \bar v_1, \cdots, \bar v_{y-1}, \bar v_y-1, \bar v_{y+1},\cdots}f(\bar{v})\\
&=w_q(n) +\sum_{\substack{\bar{v}\in\N^{J} \\ 1\leq \langle \bar{v}, J\rangle\leq n}} f(\bar{v}) \sum_{y\in J} {|\bar{v}-1| \choose  \bar v_1, \cdots, \bar v_{y-1}, \bar v_y-1, \bar v_{y+1},\cdots}\\
&=w_q(n) +\sum_{\substack{\bar{v}\in\N^{J} \\ 1\leq \langle \bar{v}, J\rangle\leq n}} f(\bar{v}) {|\bar v| \choose \bar v_1, \bar v_2, \cdots}, 
\end{align*}
by the recurrence relation for multinomial coefficients.\\
In the following we again replace \(\bar v\) by \(v\) and plug-in \(w_q(n)\) as well. Then everything comes neatly together as 
\begin{align*}
w_q(x)&= \sum_{\substack{v\in\N^{J} \\ \langle v, J\rangle\leq n}} \l(\prod_{\ell \in J} \frac{m_\ell^{v_\ell}}{v_\ell !}\r)\Big(\langle v, J\rangle-n\Big)^{|v|}
+\sum_{\substack{v\in\N^{J} \\ 1\leq\langle v, J\rangle\leq n}} \l(\prod_{\ell\in J} \frac{m_\ell^{v_\ell}}{v_\ell !}\r)\l((\langle v, J\rangle-x)^{|v|}-(\langle v, J\rangle-n)^{|v|} \r) \\
&=\sum_{\substack{v\in\N^{J} \\ \langle v, J\rangle\leq\lfloor x\rfloor}} \l(\prod_{\ell \in J} \frac{m_\ell^{v_\ell}}{v_\ell !}\r)\Big(\langle v, J\rangle-x\Big)^{|v|},
\end{align*}
which completes the proof upon realising that the continuity argument we utilised for \(x=1\) is applicable for \(x=n+1\) as well.
\end{proof}

As mentioned above we may treat lattices with arbitrary step \(\varepsilon>0\), i.e. not only integer jumps, by the scaling arguments of Lemma \ref{scaling} and Lemma \ref{scaling1}. We use the notation
\begin{align}\label{gaussepsilon}
\lfloor x \rfloor_\varepsilon:= \varepsilon \lfloor \tfrac{x}{\varepsilon} \rfloor =\max\{z\in \varepsilon\Z ~:~ z < x\}\quad \text{for all }x\in\RR, \varepsilon>0.
\end{align}

\begin{proposition}\label{general_lattice}
Let  $(L_t)_{t\geq 0}$ be a spectrally negative compound Poisson process with drift \(c>0\), intensity \(\lambda>0\) and jump measure \(\Pi=\sum_{y\in J} p_y \delta_{y}\) such that \(J:=\mathrm{supp}(\Pi)\subset \varepsilon\N\setminus\{0\}\) for some \(\varepsilon>0\). Then for \(q\geq0\) the \(q\)-scale function \(W^{(q)}\) of  $(L_t)_{t\geq 0}$ is given by
\begin{align}\label{smalllatsol1}
W^{(q)}(x)=\frac{1}{c} \e^{\frac{ q+\lambda}{c} x}\sum_{\substack{v\in\N^{J} \\ \langle v, J \rangle\leq\lfloor x \rfloor_\varepsilon}}  \l(\prod_{\ell\in J} \frac{m_{\ell}^{v_\ell}}{v_\ell !}\r)\Big(\langle v, J\rangle-x\Big)^{|v|},\quad x\geq0,
\end{align}
where \(m_y=\frac{\lambda  p_y}{c} \e^{-\frac{q+\lambda}{c} y}\), $y\in J$.
\end{proposition}
\begin{proof}
Recall from \eqref{eq-specialtype} that $L_t = c t - \sum_{i=1}^{N_t} \xi_i,$ and define the auxiliary Lévy process $(\tilde{L}_t)_{t\geq 0}$ by 
\begin{align*}
\tilde L_t = t-\sum_{i=1}^{\tilde N_{t}} \frac{\xi_i}\varepsilon, \quad t\geq 0,
\end{align*}
where \(\tilde N_t:=N_{\varepsilon t/c}\), i.e. \(\tilde L_t\) has intensity \(\tilde \lambda = \varepsilon \lambda/c\). We easily convince ourselves that the jump measure of \(\tilde L_t\) is supported on \(\tilde{J}=\frac{J}{\varepsilon}\subseteq \N\setminus\{0\}\). Thus, Lemma \ref{lattice_explicit} is applicable and we can compute the $q$-scale function of $(\tilde{L}_t)_{t\geq 0}$ via Corollary \ref{auxiliary} as
\begin{align*}
\tilde W^{(q)}(x)&= \e^{(q+\tilde{\lambda})x} \tilde{w}_q(x) 
\\
&=\e^{(q+\varepsilon\lambda/c)x}\sum_{\substack{v\in\N^{J/\varepsilon} \\ \langle v, \frac J\varepsilon \rangle\leq\lfloor x\rfloor}} \l(\prod_{\ell \in J/\varepsilon} \frac{ ({\frac{\varepsilon\lambda}{c}  p_{\varepsilon\ell}} \e^{-(q+\varepsilon\lambda/c) \ell})^{v_\ell}}{v_\ell !}\r)\Big(\langle v, \frac{J}{\varepsilon}\rangle-x\Big)^{|v|},~x\geq0,
\end{align*}
for any \(q\geq0\).
Now we consecutively apply Lemma \ref{scaling1} and Lemma \ref{scaling} to obtain for all \(q\geq0\) and \(x\geq0\) the relationship
\[
W^{(q)}(x)=\frac{1}{c} \tilde W^{(\varepsilon q/c)} \Big(\frac{x}{\varepsilon}\Big),
\]
where \(W^{(q)}\) is the scale function of \(L_t\). 
Thus, we find
\begin{align*}
W^{(q)}(x)&=\frac{1}{c} \e^{(\varepsilon q/c+\varepsilon\lambda/c)\frac{x}{\varepsilon}}\sum_{\substack{v\in\N^{J/\varepsilon} \\ \langle v, \frac J\varepsilon \rangle\leq\lfloor \frac{x}{\varepsilon} \rfloor}} \l(\prod_{\ell\in J/\varepsilon}\frac{ ({\frac{\varepsilon\lambda}{c}  p_{\varepsilon\ell}} \e^{-(\varepsilon q/c+\varepsilon\lambda/c) \ell})^{v_\ell}}{v_\ell !}\r)\Big(\langle v, \frac{J}{\varepsilon}\rangle-\frac{x}{\varepsilon}\Big)^{|v|}\\
&=\frac{1}{c} \e^{\frac{ q+\lambda}{c} x}\sum_{\substack{v\in\N^{J} \\ \langle v, J \rangle\leq\lfloor x \rfloor_\varepsilon}} \l(\prod_{\ell\in J} \frac{ ({\frac{\lambda}{c}  p_\ell} \e^{- \frac{q+\lambda}{c} \ell})^{v_\ell}}{v_\ell !}\r)\Big(\langle v, J \rangle-x \Big)^{|v|},
\end{align*}
as had to be shown.
\end{proof}

Our strategy's next step is approximation, i.e. we consider converging sequences of spectrally negative compound Poisson processes and investigate the limit of the respective sequence of scale functions. \newline
However, the formula obtained in Proposition \ref{general_lattice} is yet too cumbersome for that purpose. Luckily we can order the sum in \eqref{smalllatsol1} in a way such that the parameters of the process, i.e. drift, intensity and jump measure, become  visible.
 
\begin{corollary}\label{ordering}
	In the setting of Proposition \ref{general_lattice} it holds for all \(q\geq0\) 
	\begin{align}\label{neworder}
	W^{(q)}(x)= \frac{1}{c}  \sum_{k=0}^{\lfloor \frac x\varepsilon\rfloor} \frac{(\lambda/c)^k}{k!} \int_0^{\lfloor x\rfloor_\varepsilon}  \e^{-\frac{q+\lambda}{c} ( s - x)} ( s -x)^{k}~
	\Pi^{*k}(\diff  s), \quad x\geq 0.
	\end{align}
\end{corollary}
\begin{proof}
	Recall that \eqref{smalllatsol1} is a finite sum due to the restriction \(\langle v,J\rangle \leq \lfloor x \rfloor_\varepsilon\). But this implies \(|v|\leq \frac{1}{\varepsilon} \langle v,J\rangle \leq \lfloor \frac x\varepsilon\rfloor\) via \eqref{gaussepsilon} since \(\min\{y :~ y\in J\} \geq \varepsilon\). Thus we may reorder \eqref{smalllatsol1} according to the value of \(|v|\). We obtain for \(x\geq0\)
	\begin{align*}
	W_q(x)&=\frac{1}{c} \e^{\frac{ q+\lambda}{c} x} \sum_{k=0}^{\lfloor \frac x\varepsilon\rfloor} \sum_{\substack{v\in\N^{J} \\ \langle v, J \rangle\leq\lfloor x \rfloor_\varepsilon \\ |v|=k}}  \l(\prod_{\ell\in J} \frac{m_{\ell}^{v_\ell}}{v_\ell !}\r)\Big(\langle v, J\rangle-x\Big)^{k}\\
	&= \frac{1}{c} \e^{\frac{ q+\lambda}{c} x} \sum_{k=0}^{\lfloor \frac x\varepsilon\rfloor} \sum_{\substack{v\in\N^{J} \\ \langle v, J \rangle\leq \varepsilon \lfloor \frac{x}{\varepsilon} \rfloor \\ |v|=k}}  {k\choose v_{\varepsilon},v_{2\varepsilon},\cdots} \frac{(\lambda/c)^k}{k!} \l(\prod_{\ell \in J} p_\ell^{v_\ell}\r) \e^{-\frac{q+\lambda}{c} \langle v, J\rangle} \Big(\langle v, J\rangle-x\Big)^{k},
	\end{align*}
	where we inserted a multinomial coefficient by multiplying with \(\frac{k!}{k!}\) and set $v_x=0$ for $x\notin J$. The inner sum can be ordered even further, this time according to the value of \(\frac{1}{\varepsilon}\langle v,J\rangle\). It holds
	\begin{align*}
	W_q(x)&=\frac{1}{c} \e^{\frac{ q+\lambda}{c} x} \sum_{k=0}^{\lfloor \frac x\varepsilon\rfloor} \sum_{j=k}^{\lfloor \frac x\varepsilon\rfloor} \sum_{\substack{v\in\N^{J} \\ \langle v, J \rangle = \varepsilon j \\ |v|=k}}  {k\choose v_{\varepsilon},v_{2\varepsilon},\cdots} \frac{(\lambda/c)^k}{k!} \l(\prod_{\ell \in J} p_\ell^{v_\ell}\r) \e^{-\frac{q+\lambda}{c} \varepsilon j} \Big(\varepsilon j -x\Big)^{k}\\
	&= \frac{1}{c} \e^{\frac{ q+\lambda}{c} x} \sum_{k=0}^{\lfloor \frac x\varepsilon\rfloor} \sum_{j=k}^{\lfloor \frac x\varepsilon\rfloor} \frac{(\lambda/c)^k}{k!} \e^{-\frac{q+\lambda}{c} \varepsilon j} \Big(\varepsilon j -x\Big)^{k}
	\sum_{\substack{v\in\N^{J} \\ \langle v, J \rangle = \varepsilon j \\ |v|=k}}  {k\choose v_{\varepsilon},v_{2\varepsilon},\cdots}  \l(\prod_{\ell \in J} p_\ell^{v_\ell}\r).
	\end{align*}
	We easily convince ourselves that 
	\[
	\Pi^{*k}(\varepsilon j)=\sum_{\substack{v\in\N^{J}\\ \langle v, J\rangle=\varepsilon j \\|v|=k}} 
	{k\choose v_\varepsilon,v_{2\varepsilon},\cdots}
	\l(\prod_{\ell\in J} p_\ell^{v_\ell}\r)
	\]
	for \(j,k\in\N\). Thereby the proof is concluded since $\Pi$ is discrete and $\Pi^{\ast k}([0,k\varepsilon))=0$.
\end{proof}

	To simplify notation in the following recall  $g_k(s,x)$, $0\leq s,x$, \(k\in\N\) from \eqref{g} 
	then  \eqref{neworder} is 
	\begin{align}\label{almostready}
	W^{(q)}(x)=\frac{1}{c} \sum_{k=0}^{\lfloor \frac x\varepsilon\rfloor} \int^{\lfloor x\rfloor_\varepsilon}_0 g_k(s,x) ~\Pi^{*k}(\diff s).
	\end{align}
	Further, for \(k>\lfloor \frac{x}{\varepsilon} \rfloor\) it holds \(\mathrm{supp}(\Pi^{*k})\subset (\lfloor x \rfloor_\varepsilon,\infty)\). Moreover, for any \(x\in\R_+\setminus\N\) the interval \((\lfloor x \rfloor_\varepsilon,x]\) is a \(\Pi^{*k}\)-null set for all \(k\in\N\). Hence, we may write in the setting above
	\begin{align}\label{representation}
	W^{(q)}(x)=\frac{1}{c}  \sum_{k=0}^{\infty} \int^{x}_0 g_k(s,x) ~\Pi^{*k}(\diff s), \quad x\geq 0,
	\end{align}
	which is \eqref{mainform}.

	Now we have an expression at hand where the parameters of the Lévy process are clearly visible. Due to the isolation of the jump measure investigating limits is a straightforward task.

\subsection{Arbitrary jump distributions}\label{S33}
	In this section we make use of the fact that probability measures can be approximated by lattice measures as it is stated in the following lemma.
	\begin{lemma}\label{dense}
	The set of positive lattice distributions is dense in the set of probability distributions on $\RR_+$ with respect to the topology of weak convergence.
	\end{lemma}
	\begin{proof}
	By definition, a sequence \(\{\mu_n\}_{n\in\N}\) of probability measures converges weakly to a probability measure \(\mu\), if and only if the sequence \(\{F_n\}_{n\in\N}\) of their respective cumulative distribution functions converges pointwise to the distribution function \(F\) of \(\mu\), wherever \(F\) is continuous. \newline
	Hence, it suffices to prove that any distribution function with $F(0)=0$ can be approximated by a sequence of step functions corresponding to lattice distributions with successively smaller steps \(\varepsilon_n\searrow 0\). Such a sequence may be constructed in the following way.\newline
	For fixed \(n\in\N\) we set
	\begin{align*}
	F_n(x)=\sum_{k=0}^\infty F(k\varepsilon_n)\bone_{[k\varepsilon_n,(k+1)\varepsilon_n)}
	\end{align*}
	for all \(x\in\R_+\). By construction \(F_n\) is the cumulative distribution function of a lattice distribution as it fulfils \(F_n(0)=0\), is right-continuous and it holds \(\lim_{x\to\infty}F_n(x)=1\) due to \(\lim_{x\to\infty}F(x)=1\). Choose \(x\in\R_+\) arbitrary such that \(F\) is continuous at \(x\). For all \(n\in\N\) it then holds \(F_n(x)=F(\lfloor x \rfloor_{\varepsilon_n})\). But then \(\lim_{n\to\infty} \lfloor x \rfloor_{\varepsilon_n} = x\) implies \(\lim_{n\to\infty} F(\lfloor x \rfloor_{\varepsilon_n}) = F(x)\) and the proof is finished.
	\end{proof}

	The idea is now clear. Denote by \(L:=(L_t)_{t\geq 0}\) a spectrally negative compound Poisson process with drift \(c>0\), intensity \(\lambda>0\) and (spectrally one-sided) jump measure \(\Pi\) as in \eqref{eq-specialtype}. Choose a sequence \(\{\Pi_n\}_{n\in\N}\) of lattice distributions which converges weakly to \(\Pi\) and for \(n\in\N\) denote by \(L^{(n)}:=(L^{(n)}_t)_{t\geq 0}\) the compound Poisson process with the same drift and intensity as \(L\) but with jump measure \(\Pi_n\) instead. We keep this notation throughout this section. 
	
	By Corollary \ref{ordering} for fixed \(q\geq0\) and \(n\in\N\) the \(q\)-scale function \(W^{(q)}_n\) of \(L^{(n)}\) is
	\begin{align*}
	W^{(q)}_n(x)=\frac{1}{c}  \sum_{k=0}^{\infty} \int^{x}_0 g_k(s,x) ~\Pi_n^{*k}(\diff s), \quad x\geq 0.
	\end{align*}
	To prove that such a representation holds for the \(q\)-scale functions \(W^{(q)}\) of \((L_t)_{t\geq0}\) as well we need to show three things: First, that weak convergence of the jump measures \(\{\Pi_n\}_{n\in\N}\) implies pointwise convergence of the respective scale functions \(\{W^{(q)}_n\}_{n\in\N}\), second, that the right-hand side of \eqref{representation} is well-defined for arbitrary jump measures and third, that the limits are consistent, i.e.
	\begin{align*}
	\lim_{n\to\infty} \frac{1}{c}  \sum_{k=0}^{\infty} \int^{x}_0 g_k(s,x) ~\Pi_n^{*k}(\diff s) = \frac{1}{c}  \sum_{k=0}^{\infty} \int^{x}_0 g_k(s,x) ~\Pi^{*k}(\diff s)
	\end{align*}
	for every \(x\geq0\). 
	
	For the first issue we use the notion of equicontinuity. Recall that a family of continuous functions \(\mathcal{F}\subset \mathcal{C}(\R_+)\) is said to be \emph{equicontinuous} if and only if for all \(x\in\R_+\) and all \(\epsilon>0\) there exists \(\delta=\delta(x,\epsilon)>0\) such that for all functions \(f\in\mathcal{F}\) and all \(y\in\R_+\) it holds
	\[
	|x-y|<\delta \implies |f(x)-f(y)|<\epsilon.
	\]

	\begin{lemma}\label{equic}
	For fixed \(q\geq0\) and fixed \(c,\lambda>0\) let \(\mathcal{W}^{(q)}_{c,\lambda}\) be the set of \(q\)-scale functions of spectrally negative compound Poisson processes with parameters \(c\) and \(\lambda\) and lattice distributed jumps. Then \(\mathcal{W}^{(q)}_{c,\lambda}\) is equicontinuous.
	\end{lemma}
	\begin{proof}
	As simultaneous scaling does not affect equicontinuity we may, once again, set \(c=1\). Let \(q,\lambda\geq0\) and consider a function \(W^{(q)}\in\mathcal{W}^{(q)}_{1,\lambda}\), i.e. \(W^{(q)}\) is the \(q\)-scale function of a compound Poisson process with drift \(c=1\), intensity \(\lambda>0\) and lattice jump measure \(\Pi\) on $(0,\infty)$. We begin by proving
	\begin{align}\label{less_e}
	W^{(q)}(x) \leq \e^{(q+\lambda) x}, \quad \text{for all } x\geq 0.
	\end{align}
	Let $\varepsilon>0$ be the step of the lattice distribution $\Pi$ and \(x>\varepsilon\). As in the proof of Lemma \ref{recursivformula1} write $\sigma$ for the time of the first jump of the corresponding Lévy process, then  it follows 
	\begin{align*}
	\E^\varepsilon[\e^{-q\tau_x^+}\bone_{\{\tau_x^+<\tau_0^-\}}] 
&\geq \E^\varepsilon[\e^{-q\tau_x^+} \bone_{\{\tau_x^+<\sigma\}}] 
	= \e^{(q+\lambda) (\varepsilon-x)},
	\end{align*}
	by arguments as in the proof of Lemma \ref{recursivformula1}. Thus via \eqref{scaleviapassage}
	\begin{align*}
	W^{(q)}(x) \leq \e^{(q+\lambda) (x-\varepsilon)}W^{(q)}(\varepsilon) =  \e^{(q+\lambda) x}, \quad x>\varepsilon.
	\end{align*}
	For  \(x<\varepsilon\) we see from \eqref{smalllatsol1} that \eqref{less_e} is clearly fulfilled and hence by continuity we have shown \eqref{less_e}.\\
	Set \(W^{(q)}(x)=\e^{(q+\lambda)x}w_q(x)\) for some continuous function \(w_q\) as in Corollary \ref{auxiliary}. The inequality \eqref{less_e} then implies \(w_q(x)\leq1\) for all \(x\geq0\). Moreover, from \cite[Cor. 2.5]{kuznetsov2011theory} we know that \(w_q\in \mathcal{C}^1((0,\infty) \setminus \mathrm{supp}(\Pi))\).\\ We are going to show that
	\begin{align*}
	\l| \frac{w_q(x_2)-w_q(x_1)}{x_2-x_1}\r| \leq \lambda
	\end{align*}
	holds for all \(x_1,x_2\in\R_+\) sufficiently close. From \eqref{recurse3} (but used here for general step size $\varepsilon>0$) we know
	$$w_q(x)= w_q(\lfloor x\rfloor_\varepsilon) - \lambda\sum_{y\in J} p_y  \e^{-(q+\lambda) y}\int^x_{\lfloor x\rfloor_\varepsilon} w_q(z-y) \diff z ~\Pi(\diff y).$$
	Thus for any \(x_1,x_2\in\R_+\) such that \(\lfloor x_1 \rfloor_\varepsilon=\lfloor x_2 \rfloor_\varepsilon\)
	\begin{align*}
	\l|\frac{w_q(x_2)-w_q(x_1)}{x_2-x_1}\r| &=
	\l|\lambda \sum_{y\in J} p_y \e^{-(q+\lambda) y}\frac1{x_2-x_1}\int^{x_2}_{x_1} w_q(z-y) \diff z   \r| \\
	&\leq  \lambda \sum_{y\in J} p_y \e^{-(q+\lambda) y}\frac1{x_2-x_1}\int^{x_2}_{x_1} | w_q(z-y)|\diff z \\
	&\leq \lambda \sum_{y\in J} p_y \e^{-(q+\lambda) y} \max \{w_q(z)| z \leq x_2\}\\
	&\leq \lambda,
	\end{align*}
	i.e. $|w_q(x_2)-w_q(x_1)|\leq \lambda |x_2-x_1|$, where the Lipschitz constant $\lambda$ does not depend on the step $\varepsilon$.
	Moreover, since  \(w_q\in \mathcal{C}^1((0,\infty) \setminus \mathrm{supp}(\Pi))\), the left derivatives \(\partial_-w_q(x)\) on lattice elements \(x\in\varepsilon\N\setminus\{0\}\) are given by
	\begin{align*}
	\partial_-w_q(x)=\lim_{x_n\nearrow x} w_q'(x_n)\leq \lambda.
	\end{align*}
	Hence, the directional derivatives of \(w_q\) are bounded everywhere, and therefore the family $\{w_q: w_q(x)=e^{-(q+\lambda)x} W^{(q)}(x), W^{(q)}(x)\in \mathcal{W}^{(q)}_{c,\lambda}\}$ is equicontinuous.\\
	However, since for any $x_1, x_2\in \RR_+$
	\begin{align*}
	|W^{(q)}(x_1)-W^{(q)}(x_2)|&= \e^{(q+\lambda)x_2} | \e^{(q+\lambda)(x_1-x_2)} w_q(x_1) - w_q(x_2)|\\
	&\leq \e^{(q+\lambda)x_2} \left( |(\e^{(q+\lambda)(x_1-x_2)} -1)w_q(x_1)| + | w_q(x_1) - w_q(x_2)|\right)\\
	&\leq  \e^{(q+\lambda)x_2} \left( |(\e^{(q+\lambda)(x_1-x_2)} -1)| + | w_q(x_1) - w_q(x_2)|\right)
	\end{align*}
	this implies equicontinuity of  \(\mathcal{W}^{(q)}_{1,\lambda}\) as well.
	\end{proof}

The lemma above now allows to prove pointwise convergence of scale functions of compound Poisson processes with drift in the case that the jump distributions converge weakly. 

	\begin{lemma}\label{scale_consistent}
	For all \(x\geq0\) it holds
	\begin{align*}
	\lim_{n\to\infty}W_n^{(q)}(x)= W^{(q)}(x).
	\end{align*}
	\end{lemma}
	\begin{proof}
		By Lemma \ref{equic} the sequence \(\{W^{(q)}_n\}_{n\in\N}\) is equicontinuous. 
		Thus, by \cite[Thm. 1.7.5]{arendt2011vector}, if there are constants \(M,\omega>0\) such that for every \(n\in\N\) it holds \(|W^{(q)}_n(x)|\leq M\e^{\omega x}\) and if there exists \(c_0\geq\omega\) such that the sequence $\{ \mathcal{L}(W^{(q)}_n)(\beta) \}_{n\in\N}$ 
		converges pointwise as $n\to\infty$ for every \(\beta>c_0\), then the sequence \(\{W^{(q)}_n\}_{n\in\N}\) converges uniformly on every compact subset of \(\R_+\) as $n\to\infty$ and for every \(\beta>c_0\)
		\begin{align}\label{laplace}
		\lim_{n\to\infty} \mathcal{L}\l(W^{(q)}_n\r)(\beta) = \mathcal{L}\l(\lim_{n\to\infty}W^{(q)}_n\r)(\beta).
		\end{align}	
		Let \(\psi\) and \(\psi_n\) denote the Laplace exponents of the L\'evy processes \(L\) and \(L^{(n)}\), respectively. Then as \(\Pi_n\to \Pi\), $n\to\infty$, weakly, it follows that \(\psi_n(\beta) \to \psi(\beta)\), $n\to\infty$, for every \(\beta>\sup\{y:\psi(y)=q\}=\Phi(q)\).
		Further, from the definition of \(q\)-scale functions we know that its Laplace transforms are given by
		\[
		\mathcal{L}\l(W^{(q)}\r)(\beta) = \frac1{\psi(\beta)-q}, \quad \beta>\Phi(q),
		\] 
		 such that the above  implies pointwise convergence of $\{ \mathcal{L}(W^{(q)}_n)(\beta) \}_{n\in\N}$ for all $\beta>\Phi(q)$ as $n\to\infty$. \\
		Let \(\omega>\Phi(q)\). Since \(\mathcal{L}(W^{(q)}_n)(\omega)\in\R\) for \(n\) large enough there exist constants \(M>0\) and \(n_0\in\N\) such that \(|W_n^{(q)}(x)|\leq M\e^{\omega x}\) holds for every \(n>n_0\) and for all \(x\in\R_+\).\\
		Thus \(\{W^{(q)}_n\}_{n\in\N}\) converges uniformly on every compact subset of \(\R_+\) as $n\to\infty$, and \eqref{laplace} holds for \(\beta>\omega\). But as \(\omega>\Phi(q)\) is arbitrary, \eqref{laplace} holds for all \(\beta>\Phi(q)\), and hence the assertion follows.
	\end{proof}

 	We are now ready to give the proof of our main theorem as follows.

	\begin{proof}[Proof of Theorem \ref{main}.]
		With the previous Lemma \ref{scale_consistent}, proving well-definedness and consistency of the limits is tantamount to completing the proof.\\
		To see the well-definedness of the right hand side of \eqref{mainform}, i.e. of 
			\begin{align}\label{candidate}
		\frac{1}{c}  \sum_{k=0}^{\infty} \int^{x}_0 g_k(s,x) ~\Pi^{*k}(\diff s),
		\end{align}
		note that for fixed \(x\geq0\) a sufficient condition for the infinite series in \eqref{candidate} to converge is
	\begin{align*}
	\sum_{k=0}^{\infty} \max_{0\leq s\leq x} \l| g_k(s,x) \r| <\infty
	\end{align*}
	due to \(\Pi^{*k}\) being probability measures for every \(k\in\N\). However, by the definition of \(g_k\) in \eqref{g} we easily find constants \(C_1,C_2>0\) such that
	\begin{align*}
	\max_{0\leq s\leq x} \l| g_k(s,x) \r| < C_1\frac{C_2^k}{k!}.
	\end{align*}
	Hence, the expression \eqref{candidate} 
	is well-defined for all \(x\geq0\), and since we know 
	\begin{align}\label{w_approx}
	W^{(q)}(x)=\lim_{n\to\infty} \frac{1}{c}  \sum_{k=0}^{\infty} \int^{x}_0 g_k(s,x) ~\Pi_n^{*k}(\diff s),
	\end{align}
	from Lemma \ref{scale_consistent} it remains to prove that \eqref{candidate} coincides with the limit in \eqref{w_approx}. \\
	For fixed \(k\in\N\) we know that \(\Pi_n^{*k}\to\Pi^{*k}~(n\to\infty)\) weakly due to Lévy's continuity theorem. Indeed, the characteristic functions \(\hat\Pi_n\) of \(\Pi_n\) converge pointwise to the characteristic function \(\hat\Pi\) of \(\Pi\). Hence, we have \(\lim_{n\to\infty} \hat\Pi_n^k(x) = \hat\Pi^k(x)\) for every \(x\in\R\) and thereby \(\Pi_n^{*k}\to\Pi^{*k}~(n\to\infty)\) weakly. Thus, it follows
	\begin{align*}
	\lim_{n\to\infty} \frac{1}{c} \int^{x}_0 g_k(s,x) ~\Pi_n^{*k}(\diff s) = \frac{1}{c}\int^{x}_0 g_k(s,x) ~\Pi^{*k}(\diff s)
	\end{align*}
	for all \(x\in\R\) and every \(k\in\N\) since \(g_k\) is continuous and bounded.
	But this, the fact that \(C_1\) and \(C_2\) are independent of the jump measure, and $\lim_{k\to\infty} \Pi^{\ast k}([0,x]) \to0$ as $k\to \infty$, allow us to find \(n,m\in\N\) large enough such that for any $\epsilon>0$ 
		\begin{align*}
	\lefteqn{\Bigg| \frac{1}{c}  \sum_{k=0}^{\infty} \int^{x}_0 g_k(s,x) ~\Pi_n^{*k}(\diff s) - \frac{1}{c}  \sum_{k=0}^{\infty} \int^{x}_0 g_k(s,x) ~\Pi^{*k}(\diff s)\Bigg|}\\
	& \leq \l| \frac{1}{c}  \sum_{k=0}^{m} \int^{x}_0 g_k(s,x) ~\Pi_n^{*k}(\diff s) - \frac{1}{c}  \sum_{k=0}^{m} \int^{x}_0 g_k(s,x) ~\Pi^{*k}(\diff s)\r| \\
	&\quad +
	\l|  \frac{1}{c}  \sum_{k=m+1}^{\infty} \int^{x}_0 g_k(s,x) ~\Pi_n^{*k}(\diff s)\r|  +\l|\frac{1}{c}  \sum_{k=m+1}^{\infty} \int^{x}_0 g_k(s,x) ~\Pi^{*k}(\diff s)\r|  \\
	&\leq \epsilon+\epsilon+\epsilon,
	\end{align*}
	which finishes the proof.
	\end{proof}

\section{Primitives and Derivatives}\label{S4}
\setcounter{equation}{0}

\subsection{Smoothness of scale functions}\label{S41}

Recall that we denote by \(\overline{\Pi}\) the cumulative distribution function of the jump measure \(\Pi\) and that we abbreviate differentiation as $\partial$ and use $\partial_x:=\frac{\partial}{\partial x}$.

\begin{proposition}\label{smoothness}
For \(q\geq0\) let \(W^{(q)}\) be the \(q\)-scale function of the spectrally negative compound Poisson process in \eqref{eq-specialtype}. 
	\begin{enumerate}
	\item[(a)] [First order derivatives] We have
\begin{align}\label{D+}
\partial_{+} W^{(q)}(x)&=\frac{1}{c} \sum_{k=0}^{\infty} \int_{[0,x]} \partial_x g_k(s,x) ~\Pi^{*k}(\diff s), \quad x\geq 0,\\\label{D-}
\partial_{-} W^{(q)}(x)&=\frac{1}{c}\sum_{k=0}^{\infty} \int_{[0,x)} \partial_x g_k(s,x) ~\Pi^{*k}(\diff s), \quad x>0.
\end{align}
In particular $W^{(q)}$ is (continuously) differentiable in $x_0\geq 0$ if and only if  $\overline{\Pi}$ is continuous in $x_0$, and $W^{(q)}\in \cC^1(0,\infty)$ if and only if $\overline{\Pi}\in \cC^0(0,\infty)$.
\item[(b)] [Higher order derivatives]
For \(n\geq 1\) and  for all $x \geq 0$ such that the following exists,
\begin{align}\label{derivative}
\partial^{n+1} W^{(q)}(x)&=\frac{1}{c}\l( \sum_{k=1}^{n}\sum_{j=k}^{n}C(k,j) \partial^{n+1-j} \overline{\Pi}^{*k}(x) +  \sum_{k=0}^{\infty}  \int_{[0,x]} \partial_x^{n+1} g_k(s,x)\,\diff \overline{\Pi}^{*k}( s)\r)
\end{align}
with constants
\begin{equation} \label{propconstants}
C(k,j)=(\partial_x^j g_k(s,x))\Big|_{s=x}=\begin{cases}
\l(\frac{\lambda }{c}\r)^k{j \choose k}\l(\frac{q+\lambda}{c}\r)^{j-k}(-1)^k, &j\geq k, \\
0,& j<k,
\end{cases}, \quad k,j\in\N.
\end{equation}
In particular $W^{(q)}$ is twice (continuously) differentiable in $x_0\geq 0$ if and only if $\overline{\Pi}$ is (continuously) differentiable in $x_0$. For $n>1$, if $\overline{\Pi}\in\cC^1(0,\infty)$, then $W^{(q)}$ is \((n+1)\) times (continuously) differentiable at \(x_0\geq 0\) 
if and only if \(\overline{\Pi}\) is \(n\) times (continuously) differentiable at \(x_0\).\\
Hence \(W^{(q)} \in \cC^{n+1}(0,\infty)\) if and only if $\overline{\Pi}\in\cC^n(0,\infty)$, $n\geq 1$. 
\end{enumerate}
\end{proposition}

To prove Proposition \ref{smoothness} we will swap differentiation and summation. By an application of the Moore-Osgood Theorem (cf. \cite[Thm. 7.17]{rudin}) this is possible if the summands themselves are differentiable and if the series converges uniformly. These points will therefore be treated in the upcoming three preparatory lemmas. 

\begin{lemma}\label{diffg} Consider $g(s,x)$, $0\leq s, x$, from \eqref{g}. Then for all $k,n\in\NN$ it holds
 \begin{align} \label{dgdx}
		\partial_x^n g_k(s,x)=\l(\frac{\lambda}{c}\r)^k \e^{-\frac{q+\lambda}{c}(s-x)}\sum_{i=0}^{n\wedge k} {n\choose i} \l(\frac{q+\lambda}{c}\r)^{n-i}\frac{(s-x)^{k-i}}{(k-i)!}(-1)^i, \quad s,x\geq 0,
		\end{align}
	 such that in particular $(\partial_x^n g_k(s,x))\big|_{s=x}$ is given by \eqref{propconstants} and is independent of \(x\).
\end{lemma}
\begin{proof}
	This is immediate from the definition of $g$.
\end{proof}

\begin{lemma}\label{abcd}
	Let \(F\) be a cumulative distribution function with \(F(0)=0\).
	\begin{enumerate}
		\item[(a)] [First order derivatives] For all \(k\in\N\) 
		\begin{align}\label{lem_b1}
		\partial_{+} \int_{[0,x]} g_k(s,x)\diff F^{*k}(s)=\int_{[0,x]}\partial_x g_k(s,x) ~\diff F^{*k}(s),\quad x\geq0.
		\end{align} 
		The left derivative differs from the right derivative at \(x_0\in\R_+\) if and only if \(k=1\) and \(F(x_0-)\neq F(x_0)\). In this case it is given by
		\begin{align}\label{lem_b2}
		\partial_{-} \int_{[0,x_0]} g_1(s,x_0)\diff F(s)=\int_{[0,x_0)}\partial_x g_1(s,x_0) ~\diff F(s).
		\end{align}
		\item[(b)][Higher order derivatives] If \(F\in\C^0\), then for all \(k\in\N\), $n\in\NN$, with $C(k,n)$ as in \eqref{propconstants},
		\begin{align}\label{lem_c}
		\partial_{\pm} \int_{[0,x]} \partial_x^n g_k(s,x)\diff F^{*k}(s)&=  \int_{[0,x]}\partial^{n+1}_x g_k(s,x) ~\diff F^{*k}(s) + C(k,n) \partial_{\pm} F^{*k}(x),
		\end{align}
		for all $x\geq 0$ such that $\partial_{\pm} F^{\ast k}(x)$ exists. In particular for \(x_0\in\R_+\) such that \(\partial F^{\ast k}(x_0)\) does not exist, the function \(x\mapsto \int_{[0,x]} \partial_x^n g_k(s,x)\diff F^{\ast k}(s)\) is not differentiable at \(x_0\). 
	\end{enumerate}
\end{lemma}

\begin{proof}
	Note first that for all \(n\in\N\) it holds
	\begin{align}\label{lem_a}
	\partial_x \int_{[0,x]} \partial^n_x g_0(s,x)\diff F^{*0}(s)&= \int_{[0,x]}\partial_x^{n+1} g_0(s,x) ~\diff F^{*0}(s), \quad x\geq0,
	\end{align} due to \eqref{dgdx} and \(F^{*0}=\mathds{1}_{[0,\infty)}\). But since \eqref{lem_b1} - \eqref{lem_c} are all covered by \eqref{lem_a} if \(k=0\) we may assume \(k>0\) in the following. \\
	To prove (a) consider 
	\begin{align}
	\lefteqn{\partial_+ \int_{[0,x]} g_k(s,x)\diff F^{*k}(s)} \nonumber\\
	&=\lim_{h\to0+}\frac1{h}\l( \int_{[0,x+h]}  g_k(s,x+h) ~ \diff F^{*k}(s) -\int_{[0,x]}  g_k(s,x) ~ \diff F^{*k}(s)\r) \nonumber \\
	&=\lim_{h\to0+}\frac1{h}\l( \int_{[0,x+h]}  g_k(s,x+h) ~ \diff F^{*k}(s) -\int_{[0,x]}   g_k(s,x+h) ~ \diff F^{*k}(s) \r)  \nonumber \\
	&\quad +\lim_{h\to0+}  \int_{[0,x]} \frac{  g_k(s,x+h)-  g_k(s,x) }{h} ~ \diff F^{*k}(s) \nonumber \\
	&= \lim_{h\to0+} \frac1h\int_{(x,x+h]}  g_k(s,x+h)~\diff F^{*k}(s) + \int_{[0,x]}\partial_x g_k(s,x) ~\diff F^{*k}(s). \label{diffdecomp}
	\end{align}
	The first term vanishes since we may estimate for \(h<1\), $k\geq 1$
	\begin{align*}
	\lim_{h\to0+} \frac1h\int_{(x,x+h]}|g_k(s,x+h)|~\diff F^{*k}(s)&\leq \lim_{h\to0+} \frac{h^k}h\int_{(x,x+h]}\frac{(\lambda/c)^k }{k!}~\diff F^{*k}(s) =0,
	\end{align*}
	where for $k=1$ we use the fact that \(F^{*k}\) is right-continuous. This proves \eqref{lem_b1}. An analogue argument for the left derivative leads to
	\begin{align*}
	\partial_- \int_{[0,x]} g_k(s,x)\diff F^{*k}(s)&= \lim_{h\to0+} \frac1h\int_{(x-h,x]}g_k(s,x-h)~\diff F^{*k}(s) + \int_{[0,x]}\partial_x g_k(s,x) ~\diff F^{*k}(s).
	\end{align*}
	Again the first term vanishes for \(k>1\) regardless of the continuity of \(F\). 
	However for \(k=1\) it holds
	\begin{align*}
	\lim_{h\to0+} \frac1h\int_{(x-h,x]}g_1(s,x-h)~\diff F(s)& =\frac{\lambda }{c}(F(x)-F(x-)) = -(\partial_x g_1(s,x))\Big|_{s=x} (F(x)-F(x-))
	\end{align*}
	by \eqref{propconstants}, which proves \eqref{lem_b2}.\\
	For (b), by the same decomposition as in \eqref{diffdecomp}
	\begin{align*}
	\partial_+ \int_{[0,x]} \partial_x^n g_k(s,x)\diff F^{*k}(s)
&=	 \lim_{h\to0+} \frac1h\int_{(x,x+h]} \partial_x^n g_k(s,x+h)~\diff F^{*k}(s) + \int_{[0,x]}\partial_x^{n+1} g_k(s,x) ~\diff F^{*k}(s).\end{align*}
To determine the limit we use partial integration and obtain
\begin{align*}
\lefteqn{	\lim_{h\to0+} \frac1h\int_{(x,x+h]} \partial_x^n g_k(s,x+h)~\diff F^{\ast k}(s)} \\ 
&=\lim_{h\to0+} \frac1h \left( \left[\partial_x^n g_k(s,x+h) F^{\ast k}(s) \right]_{s=x}^{x+h} - \int_{(x,x+h]} \partial^{n+1}_x g_k(s,x+h) F^{\ast k}(s) \diff s \right)\\
&=\lim_{h\to0+} \frac{\partial_x^n g_k(x+h,x+h)\cdot F^{\ast k}(x+h) - \partial_x^n g_k(x,x+h)\cdot F^{\ast k}(x) }{h}  - \partial^{n+1}_x g_k(x,x) F^{\ast k}(x) \\
&=\partial_x^n g_k(x+h,x+h) \lim_{h\to0+} \frac{F^{\ast k}(x+h) - F^{\ast k}(x)}{h} \\
&\quad + F^{\ast k}(x)\left(\lim_{h\to 0+} \frac{ \partial_x^n g_k(x+h,x+h) - \partial_x^n g_k(x,x+h) }{h}  - \partial^{n+1}_x g_k(x,x) \right) \\
&= C(k,n)  \lim_{h\to0+} \frac{F(x+h)-F(x)}{h},
\end{align*}
since $\partial_s \partial^n_x g_k(s,x) = \partial^{n+1}_x g_k(s,x)$ as can be seen from \eqref{dgdx}. The same computation for the left derivative finishes the proof.
\end{proof}

\begin{lemma}\label{uniformconv}
	Let \(F\) be a cumulative distribution function with \(F(0)=0\) and let \(n\in\N\). Then the series
	\begin{align*}
	\sum_{k=0}^{\infty}  \int_{[0,x]} \partial_x^n g_k(s,x) ~\diff F^{*k}(s)
	\end{align*}
	converges uniformly on every compact subset of \(\R_+\).
\end{lemma}
\begin{proof}
This follows immediately from the form of  $\partial_x^n g_k(s,x)$ determined in Lemma \ref{diffg} and the fact that $F^{\ast l}$ is a cumulative distribution function for all $k\in\NN$.
\end{proof}

We are now ready to prove Proposition \ref{smoothness} by simply bringing the above arguments in the right order.

\begin{proof}[Proof of Proposition \ref{smoothness}.]
	Note first that \eqref{D+} and \eqref{D-} follow from \eqref{lem_b1} and \eqref{lem_b2} together with Lemma \ref{uniformconv} and \cite[Thm. 7.17]{rudin}. Further, from \eqref{D+} and \eqref{D-} it is clear that $W^{(q)}\in \cC^1(0,\infty)$ if and only if $\overline{\Pi}\in \cC^0(0,\infty)$. \\
To prove \eqref{derivative} we use induction on $n\in\NN$. 
Let \(n=1\). Then using first Equations \eqref{D+}, \eqref{D-}, then \cite[Thm. 7.17]{rudin}, and Lemma \ref{uniformconv}, and finally \eqref{lem_c}, 
\begin{align*}
\partial_x^2 W^{(q)}(x) &= \partial_x \left(\frac{1}{c} \sum_{k=0}^\infty \int_{[0,x]} \partial_x g_k(s,x) \,\Pi^{\ast k} (\diff s) \right)\\
&= \frac{1}{c} \sum_{k=0}^\infty \partial_x\l( \int_{[0,x]} \partial_x g_k(s,x) \,\Pi^{\ast k} (\diff s)\r)\\
&= \frac{1}{c} \left(\sum_{k=0}^\infty\int_{[0,x]} \partial_x^2 g_k(s,x) \Pi^{\ast k} (\diff s)  + C(1,1) \, \partial \overline{\Pi}(x) \right),
\end{align*}
since $C(k,1)=0$ for all $k>1$. This is \eqref{derivative} for $n=1$. Furthermore it follows immediately from the computation above, that $W^{(q)}$ is twice (continuously) differentiable in $x_0$ if and only if $\overline{\Pi}$ is (continuously) differentiable in $x_0$. Hence $W^{(q)}\in \cC^2(0,\infty)$ if and only if $\overline{\Pi}\in \cC^1(0,\infty)$.\\
Now let $n>1$ and assume that \eqref{derivative} holds for $n-1$. Then by differentiating the two terms separately and assuming existence of all appearing derivatives
\begin{align*}
\partial^{n+1} W^{(q)}(x)&=\frac{1}{c}\l(\partial_x\l( \sum_{k=1}^{n-1}\sum_{j=k}^{n-1}C(k,j) \partial^{n-j} \overline{\Pi}^{*k}(x)\r) +  \partial_x\l(\sum_{k=0}^{\infty}  \int_{[0,x]} \partial_x^n g_k(s,x) ~\diff \overline{\Pi}^{*k}( s)\r)\r),
\end{align*}
where we obtain for the first term simply
\begin{align*}
\partial_x \sum_{k=1}^{n-1}\sum_{j=k}^{n-1}C(k,j) \partial^{n-j} \overline{\Pi}^{*k}(x) = \sum_{k=1}^{n-1}\sum_{j=k}^{n-1}C(k,j) \partial^{n-j+1} \overline{\Pi}^{*k}(x).
\end{align*}
For the second term using \cite[Thm. 7.17]{rudin}, Lemma \ref{uniformconv}, and \eqref{lem_c} as before we obtain
\begin{align*}
\partial_x\sum_{k=0}^{\infty}  \int_{[0,x]} \partial_x^n g_k(s,x) ~\diff \overline{\Pi}^{*k}( s) &=  \sum_{k=0}^{\infty}  \l(   \int_{[0,x]} \partial_x^{n+1} g_k(s,x) ~\diff \overline{\Pi}^{*k}( s)+ C(k,n)\, \partial \overline{\Pi}^{*k}(x) \r)\\
&= \sum_{k=0}^{\infty}    \int_{[0,x]} \partial_x^{n+1} g_k(s,x) ~\diff \overline{\Pi}^{*k}( s)+ \sum_{k=1}^{n}  C(k,n) \,\partial \overline{\Pi}^{*k}(x),
\end{align*}
since the second term vanishes for all but finitely many \(k\in\N\) due to \eqref{propconstants}.
Adding the two derivatives together we finally obtain
\begin{align*}
\partial^{n+1} W^{(q)}(x)&=\frac{1}{c}\l(\sum_{k=1}^{n}\sum_{j=k}^{n}C(k,j) \partial^{n+1-j} \overline{\Pi}^{*k}(x) +  \sum_{k=0}^{\infty}  \int_{[0,x]} \partial_x^{n+1} g_k(s,x) ~\diff \overline{\Pi}^{*k}( s)\r).
\end{align*}
This proves \eqref{derivative} for $n$ and exactly all $x\geq 0$ such that $\partial^{j} \overline{\Pi}^{*k}(x)$ exists for $k= 1,\ldots, n$, $j=1, \ldots, n-k+1$. \\
Whenever $n>1$ and $\overline{\Pi} \in \cC^1(0,\infty)$ then there exists a continuous density $\pi$ of $\overline{\Pi}$ such that $\Pi^{\ast k}(\diff x) = \pi^{\ast k}(x) \diff x$ and $\overline{\Pi}^{\ast k}(x)= \int_0^x \pi^{\ast k}(y) dy$.
 This, however, implies that $\overline{\Pi}^{\ast k}$, $k\geq 1$, is $j$ times (continuously) differentiable in $x_0$ if $\overline{\Pi}$ is $j$ times (continuously) differentiable in $x_0$, since
$$\partial^{j} \overline{\Pi}^{*k}(x_0) = \partial^{j-1} \pi^{\ast k} (x_0) = \pi^{\ast (k-1)} \ast ( \partial^{j-1}\pi ) (x_0),\quad  k\geq 1.$$
Together with the above we observe that in this case $W^{(q)}$ is \((n+1)\) times (continuously) differentiable at \(x_0\geq 0\) if and only if \(\overline{\Pi}\) is \(n\) times (continuously) differentiable at \(x_0\).\\
The final conclusion $W^{(q)}\in \cC^{n+1}(0,\infty)$ if and only if $\overline{\Pi}\in \cC^n(0,\infty)$ is now immediate.
\end{proof}

\subsection{Primitives of q-scale functions}\label{S42}

Considering lattice distributed jumps it is easy to compute the primitives of the \(q\)-scale functions from the representation \eqref{representation} since it consists of only finitely many summands as shown in Lemma \ref{lem_primderiv}. The subsequent approximation arguments which allow for a general formula as given in Proposition \ref{diff_and_int} below, are almost the same as the ones used in Section \ref{S33} and we constrain ourselves on carving out the differences.

\begin{lemma}\label{lem_primderiv}
	In the setting of Proposition \ref{general_lattice} it holds
	\begin{align}\label{prim}
	\int_0^x W^{(q)}(y) \diff y &=\frac{1}{c} \sum_{k=0}^{\infty} \int_{[0,x)}\int_s^x g_k(s,y) \diff y ~\Pi^{*k}(\diff s), \quad x\geq 0.
	\end{align}
\end{lemma}
\begin{proof}
	We prove \eqref{prim} by induction. For \(x\in[0,\varepsilon)\) it holds
	\[
	W^{(q)}(x)= g_0(0,x).
	\]
	Thus, \eqref{prim} is valid for \(x\in[0,\varepsilon)\) and by continuity for \(x=\varepsilon\) as well. Suppose that it holds on \([0,z]\) for some \(z\in\varepsilon\N\) and let \(x\in(z,z+\varepsilon)\). We begin by decomposing
	\[
	\int_0^x W^{(q)}(y) \diff y = \int_0^z W^{(q)}(y) \diff y+\int_z^x W^{(q)}(y) \diff y.
	\]
	Using \eqref{almostready} we may write the second term as
	\begin{align*}
	\int_z^x W^{(q)}(y) \diff y &= \int_z^x \sum_{k=0}^{z/\varepsilon} \int_{[0,z]} g_k(s,y) ~\Pi^{*k}(\diff s) \diff y=  \sum_{k=0}^{z/\varepsilon} \int_{[0,z]} \int_z^x g_k(s,y)\diff y ~\Pi^{*k}(\diff s).
	\end{align*}
 Together with the induction hypothesis we obtain
	\begin{align*}
	\int_0^x W^{(q)}(y) \diff y  &=  \sum_{k=0}^{z/\varepsilon} \int_{[0,z)}\int_s^z g_k(s,y) \diff y ~\Pi^{*k}(\diff s)+\sum_{k=0}^{z/\varepsilon} \int_{[0,z]} \int_z^x g_k(s,y)\diff y ~\Pi^{*k}(\diff s)\\
	&=\sum_{k=0}^{z/\varepsilon}\int_{[0,z]}\int_s^x g_k(s,y) \diff y ~\Pi^{*k}(\diff s) \\
	&= \sum_{k=0}^{\infty}\int_{[0,x)}\int_s^x g_k(s,y) \diff y ~\Pi^{*k}(\diff s).
	\end{align*}
	Thus, \eqref{prim} holds on \([0,z+\varepsilon]\) by continuity and by induction on \(\R_+\).
\end{proof}

The obtained formula is also valid in the general context as shown by the following proposition. 

\begin{proposition}\label{diff_and_int}
	For \(q\geq0\) let \(W^{(q)}\) be the \(q\)-scale function of the spectrally negative compound Poisson process in \eqref{eq-specialtype}. Then its primitive is given by \eqref{prim}.
\end{proposition}

\begin{proof}
	As mentioned above, the proof of \eqref{prim} for arbitrary jump distributions is almost identical to the approximation part of the proof of Theorem \ref{main}. Indeed, it remains to prove that weak convergence of the jump measures \(\{\Pi_n\}_{n\in\N}\) implies pointwise convergence of the respective functions \(\int W_n^{(q)}(y) \diff y\) since both well-definedness and consistency can be shown in exactly the same way as for \(W^{(q)}\) itself.\\
	Let \(L, W^{(q)}\) and \(L^{(n)}, W^{(q)}_n, n\in\N\) be as in Section \ref{S33}. By Lemma \ref{scale_consistent} we already know that \(W^{(q)}_n \underset{n\to\infty} \longrightarrow W^{(q)}\) uniformly on any compact subset of \(\R_+\). That, however, immediately implies
	\begin{align*}
	\lim_{n\to\infty} \int_0^x W_n^{(q)}(y) \diff y = \int_0^x W^{(q)}(y) \diff y
	\end{align*}	 
	for \(x\geq0\) by dominated convergence, and hence \eqref{prim} holds for arbitrary jump measures.
\end{proof}

\subsection{Further representations}\label{S43}

	The formulae for $W^{(q)}$ in Theorem \ref{main}, its directional first derivatives in Proposition \ref{smoothness} and its primitive  in Proposition \ref{diff_and_int} share the form
	\begin{align*}
	V(x):=\sum_{k=0}^{\infty} \frac{\lambda^k}{k!} \int_0^\infty f_k(s,x) ~\Pi^{*k}(\diff s)
	\end{align*}
	for some functions \(f_k\). Since for \(k\in\N\) it holds \(\sum_{i=0}^{k}\xi_i \sim \Pi^{*k}\) and \(\P(N_{t}=k)=\frac{(\lambda t)^k}{k!}\e^{-\lambda t}\) it thus follows that
	\begin{align*}
	V(x)&= \e^{\lambda t}\sum_{k=0}^\infty \E\l[t^{-k}f_{N_t}\l(\sum_{i=0}^{N_t}\xi_i,x\r)~\middle|~N_t=k \r]\P(N_t=k) =\e^{\lambda t}\E\l[t^{-N_t}f_{N_t}\l(\sum_{i=0}^{N_t}\xi_i ,x\r)\r].
	\end{align*}
	For example, to obtain a representation for the $q$-scale function itself we have to choose $$f(s,x)=\frac{1}{c} \frac{k!}{\lambda^k} g_k(s,x) \mathds{1}_{[0,x]}(s) =\frac{1}{c^{k+1}} \e^{-\frac{q+\lambda}{c}(s-x)} (s-x)^k \mathds{1}_{[0,x]}(s) $$ which yields 
	\begin{align*}
	W^{(q)}(x)&= \frac{\e^{\lambda t}}{c} \EE\left[ (c t)^{-N_t} \e^{-\frac{q+\lambda}{c}(\sum_{i=0}^{N_t}\xi_i-x)} \l(\sum_{i=0}^{N_t}\xi_i-x\r)^{N_t}  \mathds{1}_{\sum_{i=0}^{N_t}\xi_i \leq x} \right]\\	
	&=\frac{\e^{\lambda t}}{c} \E^x \l[ \e^{-\frac{q+\lambda}{c}(c t-L_t)}\l(\frac{c t-L_t}{c t}\r)^{N_t} \bone_{\{L_t \geq c t \}} \r],
	\end{align*}
	for every \(q,x,t\geq0\). In particular we may choose $t=1$ to get
	\begin{equation}\label{eq-qscaleexpvalue}
	W^{(q)}(x) = \frac{\e^{-q}}{c} \E^x\l[ \e^{\frac{q+\lambda}{c} L_1}\l(1-\frac{L_1}{c}\r)^{N_1} \bone_{\{L_1\geq c\}} \r], \quad q,x\geq 0.
	\end{equation} 
	
	The following corollary collects the corresponding representations for the derivatives and primitives of the $q$-scale functions. The proof works as the one given for \eqref{eq-qscaleexpvalue} and is omitted.

\begin{corollary}\label{expected_values}
	For \(q\geq0\) let \(W^{(q)}\) be the \(q\)-scale function of the spectrally negative compound Poisson process in \eqref{eq-specialtype}, then for all $x\geq 0$
	\begin{align*}
	\partial_{+} W^{(q)}(x)&= \frac{\e^{-q}}{c^2} \E^x\l[\e^{\frac{q+\lambda}{c} L_1} \l(1-\frac{L_1}{c}\r)^{N_{1}-1}\l( (q+\lambda)\l(1-\frac{L_1}{c}\r) - N_1\r)\bone_{\{L_1 \geq c \}}\r],\\
	\partial_{-} W^{(q)}(x)&=\frac{\e^{-q}}{c^2} \E^x\l[\e^{\frac{q+\lambda}{c} L_1} \l(1-\frac{L_1}{c}\r)^{N_{1}-1}\l( (q+\lambda)\l(1-\frac{L_1}{c}\r) - N_1\r)\bone_{\{L_1 > c \}}\r], \\
	\int_0^x W^{(q)}(y)\diff y &= \e^\lambda \E^x\l[ \int_0^{\frac{L_1}{c}-1} \e^{(q+\lambda)y} (-y)^{N_1} \diff y \bone_{\{L_1 >c\}}  \r].\\
	\end{align*}
\end{corollary}

\section{Examples}\label{S5}
\setcounter{equation}{0}

	As discussed in \cite{hubalek2010} it is of great use to have examples of spectrally negative Lévy processes at hand for which the \(q\)-scale functions can be computed explicitly. In this section we apply our results to expand the library of those cases.\\
	In the following let \((L_t)_{t\geq0}\) be as in \eqref{eq-specialtype}. We begin by presenting two discrete examples for which the convolutions are easy to find, namely geometrically distributed and zero-truncated Poisson distributed jumps.

	\begin{example}[Geometrically distributed jumps]
	For fixed \(p\in(0,1)\) let 
$\Pi(\cdot)= \sum_{n=1}^\infty (1-p)^{n-1}p\, \delta_{n}(\cdot)$
	be a geometric distribution, then it is well-known that the measure \(\Pi^{*k}\) describes a negative binomial distribution, i.e.
	\begin{align*}
	\Pi^{\ast k}(\cdot)=\sum_{n=k}^\infty { n-1 \choose n-k} (1-p)^{n-k}p^k\, \delta_n(\cdot), \quad k\in \NN.
	\end{align*}
	Hence for \(q\geq0\) the \(q\)-scale function of \((L_t)_{t\geq0}\) follows from Theorem \ref{main} as 
	\begin{align*}
	W^{(q)}(x)&= \frac{1}{c} \sum_{k=0}^{\infty} \int_{[0,x]} g_k(s,x) ~\Pi^{*k}(\diff s) 
	= \frac{1}{c} \sum_{k=0}^{\lfloor x \rfloor} \sum_{n=k}^{\lfloor x \rfloor} g_k(n,x) { n-1 \choose n-k} (1-p)^{n-k}p^k \\
	&= \frac{1}{c}  \sum_{k=0}^{\lfloor x \rfloor} \sum_{n=k}^{\lfloor x \rfloor}\frac{(\lambda/c)^k }{k!} { n-1 \choose n-k} \e^{-\frac{q+\lambda}{c}(n-x)} (1-p)^{n-k}(p(n-x))^k,\qquad x\geq 0.
	\end{align*}
	\end{example} 
	
	\begin{example}[Zero-truncated Poisson distributed jumps]
	Here we denote by \(\Pi\) the distribution of a Poisson random variable \(\xi\) with parameter \(\mu>0\), given \(\xi>0\). To apply Theorem \ref{main} the probability mass functions \(f_k\) of \(\Pi^{*k}, k\in\N\), are needed. For \(k\neq0\) they are given by
	\begin{align}\label{ztp}
	f_k(n)=\frac{\P\l(\sum_{i=1}^k \xi_i = n \cap \forall i \leq k : \xi_i>0 \r)}{\P\l(\forall i \leq k : \xi_i >0\r)}=: \frac{z(n,k)}{(1-\e^{-\mu})^k} , \quad n\in\NN\setminus\{0\},
	\end{align}
	where \(\xi_i, i\in\N\setminus\{0\}\) are i.i.d. copies of \(\xi\). For \(k=0\) we have by definition \(\Pi^{*0}=\delta_0\). The numerator \(z(n,k)\) in \eqref{ztp} can be computed recursively as follows. Obviously it holds \(z(n,k)=0\) if \(n<k\in\N\setminus\{0\}\) and \(z(n,1)=\frac{\mu^s}{s!}\e^{-\mu}\) for \(n\in\N\setminus\{0\}\). Everywhere else we have
	\begin{align*}
	z(n,k)&=\P\Bigg(\sum_{i=1}^k  \xi_i = n\Bigg)- \sum_{\ell=1}^{k-1}{k\choose \ell} \P\Bigg(\sum_{i=1}^{\ell} \xi_i =0\Bigg)z(n,k-\ell)\\
	&= \frac{(k\mu)^n}{n!}\e^{-k\mu} - \sum_{\ell=1}^{k-1} {k \choose \ell}\e^{-\ell\mu}z(n,k-\ell).
	\end{align*}
	Further setting $z(0,0)=1$ and $z(n,0)=0, n\in\NN\setminus\{0\}$, the $q$-scale functions of $(L_t)_{t\geq 0}$ are then given by
	$$W^{(q)}(x)=\frac{1}{c}  \sum_{k=0}^{\lfloor x \rfloor} \sum_{n=k}^{\lfloor x \rfloor}\frac{(\lambda/c)^k }{k!} \e^{-\frac{q+\lambda}{c}(n-x)} \l(\frac{n-x}{1-\e^{-\mu}}\r)^k z(n,k),\qquad x\geq 0.$$
	\end{example} 
	
	This section would be incomplete without illustrating the applicability of our results for continuous jump distributions. In \cite{Egami2010PhasetypeFO} explicit representations of the $q$-scale functions have been derived for Lévy processes with phase-type distributed jumps, thus including compound Poisson processes with Erlang distributed jumps. Theorem \ref{main} now admits a representation for those processes where the jumps follow an arbitrary Gamma distribution.
	
	\begin{example}[Gamma distributed jumps]
	For \(i=1,2,\ldots\) let \(\xi_i\sim\Gamma(\alpha,\beta)\) with parameters \(\alpha,\beta>0\), i.e. it holds
	\[
	\Pi(\diff s)= \frac{\beta^\alpha}{\Gamma(\alpha)}s^{\alpha-1}\e^{-\beta s} \bone_{\{s\geq0\}}\diff s,
	\]
and 
	\[
	\Pi^{*k}(\diff s)=\frac{\beta^{k\alpha}}{\Gamma(k\alpha)}s^{k\alpha-1}\e^{-\beta s} \bone_{\{s\geq0\}} \diff s, \quad k\in\N\setminus\{0\}, 
	\]
	and \(\Pi^{*0}=\delta_0\). Applying Theorem \ref{main} to obtain the \(q\)-scale function of \((L_t)_{t\geq0}\) we compute
	\begin{align*}
	W^{(q)}(x)&=\e^{\frac{q+\lambda}{c } x}+\sum_{k=1}^{\infty} \int_0^x \frac{(\lambda/c)^k}{k!}\e^{-\frac{q+\lambda}{c }(s-x)}(s-x)^k \frac{\beta^{k\alpha}}{\Gamma(k\alpha)} s^{k\alpha-1}\e^{-\beta s}  \diff s \\
	&=\e^{\frac{q+\lambda}{c } x}\Big( 1+\sum_{k=1}^{\infty}\frac{(\lambda/c)^k}{k!}\frac{\beta^{k\alpha}}{\Gamma(k\alpha)} \int_0^x \e^{-(\frac{q+\lambda}{c} +\beta)s}(s-x)^k s^{k\alpha-1} \diff s\Big).
	\end{align*}
	In the following we denote \(C_k:=\frac{(\lambda/c)^k}{k!}\frac{\beta^{k\alpha}}{\Gamma(k\alpha)}\) and \(\rho:=\frac{q+\lambda}{c}+\beta\). Applying the binomial theorem we obtain
	\begin{align*}
	W^{(q)}(x)&=\e^{\frac{q+\lambda}{c} x}\l(1+\sum_{k=1}^{\infty}C_k \sum_{\ell=0}^{k}{k\choose \ell}(-x)^{k-\ell}\int_0^x \e^{-\rho s} s^{k\alpha+\ell-1} \diff s\r) \\
	&=\e^{\frac{q+\lambda}{c} x}\l(1+\sum_{k=1}^{\infty}C_k \sum_{\ell=0}^{k}{k\choose \ell}(-x)^{k-\ell} \rho^{-(k\alpha+\ell)}\gamma(k\alpha+\ell,\rho x)\r), \quad x\geq 0,
	\end{align*}
	where \(\gamma(\cdot,\cdot)\) is the lower incomplete gamma function.
	\end{example}

 \bibliography{scalefunction}

\begin{thebibliography}{10}

\bibitem{arendt2011vector}
W.~Arendt, C.J.K. Batty, M.~Hieber, and F.~Neubrander.
\newblock {\em Vector-valued Laplace Transforms and Cauchy Problems}.
\newblock Springer, 2nd edition, 2011.

\bibitem{asmussenalbrecher}
S.~Asmussen and H.~Albrecher.
\newblock {\em Ruin probabilities}.
\newblock World Scientific, 2nd edition, 2010.

\bibitem{Bertoin97}
J.~Bertoin.
\newblock Exponential decay and ergodicity of completely asymmetric {L}\'evy
  processes in a finite interval.
\newblock {\em Ann. Appl. Probab.}, 7:156--169, 1997.

\bibitem{chankypsavov}
T.~Chan, A.E. Kyprianou, and M.~Savov.
\newblock Smoothness of scale functions for spectrally negative {L}\'evy
  processes.
\newblock {\em Probab. Theory Relat. Fields}, 150:691--708, 2011.

\bibitem{doney}
R.~Doney.
\newblock Fluctuation theory for {L}\'evy processes.
\newblock In Jean Picard, editor, {\em Lecture Notes in Mathematics}, volume
  1897. Springer, Berlin, 2007.

\bibitem{Egami2010PhasetypeFO}
M.~Egami and K.~Yamazaki.
\newblock Phase-type fitting of scale functions for spectrally negative
  {L}{\'e}vy processes.
\newblock {\em J. Comput. Appl. Math.}, 264:1--22, 2014.

\bibitem{emery}
D.J. Emery.
\newblock Exit problem for a spectrally positive process.
\newblock {\em Adv. Appl. Probab.}, 5:498--520, 1973.

\bibitem{Erlang1909}
A.K. Erlang.
\newblock Sandsynlighedsregning og telefonsamtaler.
\newblock {\em Nyt Tidsskrift for Matematik}, B20:33--39, 1909.
\newblock Reprinted 1948 as 'The theory of probabilities and telephone
  conversations' in \textit{The Life and work of A.K. Erlang} \textbf{2},
  131-137, Trans. Danish Academy Tech. Sci.

\bibitem{hubalek2010}
F.~Hubalek and A.~E. Kyprianou.
\newblock Old and new examples of scale functions for spectrally negative
  {L}\'evy processes.
\newblock In {\em Sixth Seminar on Stochastic Analysis, Random Fields and
  Applications}, Progress in Probability, pages 119--146. Birkh\"auser, 2010.

\bibitem{kuznetsov2011theory}
A.~Kuznetsov, A.E. Kyprianou, and V.~Rivero.
\newblock The theory of scale functions for spectrally negative {L}\'evy
  processes.
\newblock In {\em L\'evy Matters II}, Springer Lecture Notes in Mathematics.
  2013.

\bibitem{Kyprianou2014}
A.~E. Kyprianou.
\newblock {\em Fluctuations of {L}\'evy processes with {A}pplications}.
\newblock Springer, 2nd edition, 2014.

\bibitem{KyprianouGS}
A.E. Kyprianou.
\newblock {\em {G}erber-{S}hiu {R}isk {T}heory}.
\newblock Springer, 2013.

\bibitem{kyprivsong}
A.E. Kyprianou, V.~Rivero, and R.~Song.
\newblock Convexity and smoothness of scale functions and de {F}inetti's
  control problem.
\newblock {\em J. Theor. Probab.}, 23:547--564, 2010.

\bibitem{rogers1}
L.C.G. Rogers.
\newblock The two-sided exit problem for spectrally positive {L}\'evy
  processes.
\newblock {\em Adv. Appl. Probab.}, 22:486--487, 1990.

\bibitem{rogers2}
L.C.G. Rogers.
\newblock Evaluating first-passage probabilities for spectrally one-sided
  {L}\'evy processes.
\newblock {\em J. Appl. Probab.}, 37:1173--1180, 2000.

\bibitem{rudin}
W.~Rudin.
\newblock {\em Principles of mathematical analysis}.
\newblock McGraw-Hill, Inc., 3rd edition, 1964.

\end{thebibliography}
	  
 \end{document}